\documentclass[a4paper,11pt]{amsart}
\pdfoutput=1

\usepackage{graphicx}
\usepackage{amsmath}
\usepackage{amssymb}
\usepackage[active]{srcltx}
\usepackage{hyperref}

\newtheorem{thm}{Theorem}%
\newtheorem{lem}{Lemma}%
\newtheorem{prop}{Proposition}%
\theoremstyle{definition}

\theoremstyle{remark}
\newtheorem{remark}{Remark}[section] %

\theoremstyle{plain}

\numberwithin{equation}{section}

\def\MM{{\mathbb M}}
\def\NN{{\mathbb N}}

\def\RR{{\mathbb R}}

\def\ZZ{{\mathbb Z}}

\def\hatNN{\widehat{\mathbb N}}

\def\S{\operatorname{S{}}}

\def\veca{{\text{\boldmath$a$}}}

\def\vecb{{\text{\boldmath$b$}}}

\def\vecc{{\text{\boldmath$c$}}}

\def\vece{{\text{\boldmath$e$}}}

\def\vecell{{\text{\boldmath$\ell$}}}
\def\vecm{{\text{\boldmath$m$}}}
\def\vecn{{\text{\boldmath$n$}}}
\def\vecq{{\text{\boldmath$q$}}}

\def\vecp{{\text{\boldmath$p$}}}
\def\vecr{{\text{\boldmath$r$}}}

\def\vecs{{\text{\boldmath$s$}}}

\def\vecv{{\text{\boldmath$v$}}}

\def\vecx{{\text{\boldmath$x$}}}

\def\vecy{{\text{\boldmath$y$}}}

\def\vecz{{\text{\boldmath$z$}}}
\def\vecalf{{\text{\boldmath$\alpha$}}}

\def\veceps{{\text{\boldmath$\epsilon$}}}

\def\veczeta{{\text{\boldmath$\zeta$}}}

\def\vecnull{{\text{\boldmath$0$}}}

\def\scrD{{\mathcal D}}
\def\scrE{{\mathcal E}}

\def\fF{{\mathfrak F}}

\def\fP{{\mathfrak P}}

\def\diag{\operatorname{diag}}

\def\diam{\operatorname{diam}}

\def\scl{\operatorname{scl}}

\def\GL{\operatorname{GL}}

\def\SL{\operatorname{SL}}

\def\SO{\operatorname{SO}}

\def\sgn{\operatorname{sgn}}

\def\vol{\operatorname{vol}}

\def\trans{\,^\mathrm{t}\!}

\newcommand{\R}{\mathbb{R}}
\newcommand{\Z}{\mathbb{Z}}

\newcommand{\minmod}{\text{ mod }}

\newcommand{\sfrac}[2]{{\textstyle \frac {#1}{#2}}}
\newcommand{\col}{\: : \:}

\newcommand{\bn}{\mathbf{0}}

\newcommand{\ve}{\varepsilon}

\newcommand{\matr}[4]{\left( \begin{matrix} #1 & #2 \\ #3 & #4 \end{matrix} \right) }
\newcommand{\smatr}[4]{\left( \begin{smallmatrix} #1 & #2 \\ #3 & #4 \end{smallmatrix} \right) }

\title{Diameters of random circulant graphs}
\author{Jens Marklof}
\author{Andreas Str\"ombergsson}
\address{School of Mathematics, University of Bristol,
Bristol BS8 1TW, U.K.\newline
\rule[0ex]{0ex}{0ex} \hspace{8pt}{\tt j.marklof@bristol.ac.uk}}
\address{Department of Mathematics, Box 480, Uppsala University,
SE-75106 Uppsala, Sweden\newline
\rule[0ex]{0ex}{0ex} \hspace{8pt}{\tt astrombe@math.uu.se}}
\date{\today}



\thanks{J.M.\ is supported by a Royal Society Wolfson Research Merit Award and a Leverhulme Trust Research Fellowship.
A.S.\ is a Royal Swedish Academy of Sciences Research Fellow supported by
a grant from the Knut and Alice Wallenberg Foundation, and is furthermore supported by the Swedish Research Council Grant 621-2007-6352.}

\begin{document}

\begin{abstract}
The diameter of a graph measures the maximal distance between any pair of vertices. The diameters of many small-world networks, as well as a variety of other random graph models, grow logarithmically in the number of nodes. In contrast, the worst connected networks are cycles whose diameters increase linearly in the number of nodes. In the present study we consider an intermediate class of examples: Cayley graphs of cyclic groups, also known as circulant graphs or multi-loop networks. We show that the diameter of a random circulant $2k$-regular graph with $n$ vertices scales as $n^{1/k}$, and establish a limit theorem for the distribution of their diameters. We obtain analogous results for the distribution of the average distance and higher moments.
\end{abstract}

\maketitle

\section{Introduction \label{secIntro}}

The diameter of a graph is the largest distance between any pair of vertices, and is a popular measure for the connectedness of a network. Many models of small-world networks, for example, have diameters that grow slowly (i.e., logarithmically) with the total number of nodes \cite{Bollobas04}, \cite{aGfX2007}. The same phenomenon is observed for a wide variety of other random graph models, and has been proved rigorously in many instances \cite{Bollobas81}, \cite{Bollobas82}, \cite{Chung01}, \cite{Ding10}, \cite{Fernholz07}, \cite{Lu01}, \cite{Riordan2010}. The worst connected networks are cycles, whose diameters increase linearly with the number of vertices. Here, connectedness is dramatically improved by additionally linking every vertex with a random partner; the logarithmic growth of the diameter is then recovered \cite{Bollobas88}. 

In the present paper we consider a more regular generalization, the {\em circulant graphs} (often also called {\em multi-loop networks}) which comprise an interwoven assembly of cycles (Figs.\ \ref{figCG8}, \ref{figCG10} left). We will show that the diameter of a random $2k$-regular circulant graph with $n$ vertices scales as $n^{1/k}$, and prove a limit theorem for the distribution of diameters of such graphs; the existence of a limit distribution was recently conjectured in \cite{Amir10}. Analogous results hold for the distribution of the average distance in a circulant graph and related quantities, see Sec.\ \ref{secEx} for details. It is interesting to note that an algebraic scaling  of the diameter has also been observed for the largest connected component of the critical Erd\"os-R\'enyi random graph \cite{Nachmias08}; here the scaling factor is $n^{1/3}$.

We furthermore establish corresponding results for circulant digraphs (cf.\ Figs.\ \ref{figCG8}, \ref{figCG10} right), where the limit distribution of diameters turns out to coincide with the limit distribution of Frobenius numbers in $d=k+1$ variables studied in \cite{Marklof10}. The connection of these two objects has been exploited previously %
\cite{dBjHaNsW2005}, \cite{aN79}, \cite{Rodseth96}, \cite{Ustinov10}. 
As for the Frobenius problem \cite{Kannan92}, the question of calculating the diameter of circulant graphs can be transformed to a problem in the geometry of numbers \cite{Cai99}, \cite{Zerovnik93}. We will use a particularly transparent approach that identifies circulant graphs with lattice graphs on flat tori \cite{sCjSmAtC2010}, \cite{rDvF2004}, and then employ the ergodic-theoretic method developed in \cite{Marklof10} to prove the existence of the limit distribution of diameters. 

Let us fix an integer vector $\veca=(a_1,\ldots,a_k)$ with distinct positive coefficients $0<a_1<\ldots<a_k\leq \frac{n}{2}$. We construct a graph $C_n(\veca)$ with $n$ vertices $0,1,2,\ldots,n-1$, by connecting vertex $i$ and $j$ whenever $|i-j|\equiv a_h \bmod n$ for some $h\in\{ 1,\ldots, k\}$. Because the adjacency matrix of this graph is circulant, $C_n(\veca)$ is called a {\em circulant graph}.  If $a_k<\frac{n}{2}$, then $C_n(\veca)$ is $2k$-regular, i.e., every vertex has precisely $2k$ neighbours. If $a_k=\frac{n}{2}$, then $C_n(\veca)$ is $(2k-1)$-regular.
It is easy to see that $C_n(\veca)$ is connected if and only if $\gcd(a_1,\ldots,a_k,n)=1$. 
In this case $C_n(\veca)$ is the (undirected) Cayley graph of $\Z/n\Z$ with respect to the generating set $\{\pm a_1,\ldots,\pm a_k\}$.

\begin{figure}
\begin{center}
\includegraphics[width=0.3\textwidth]{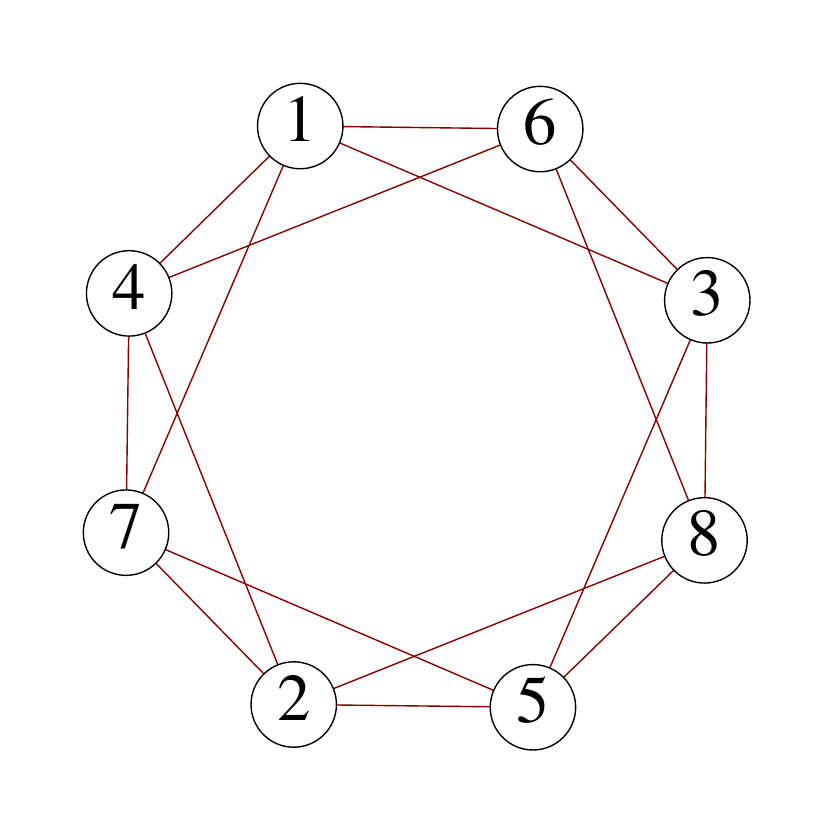}
\includegraphics[width=0.3\textwidth]{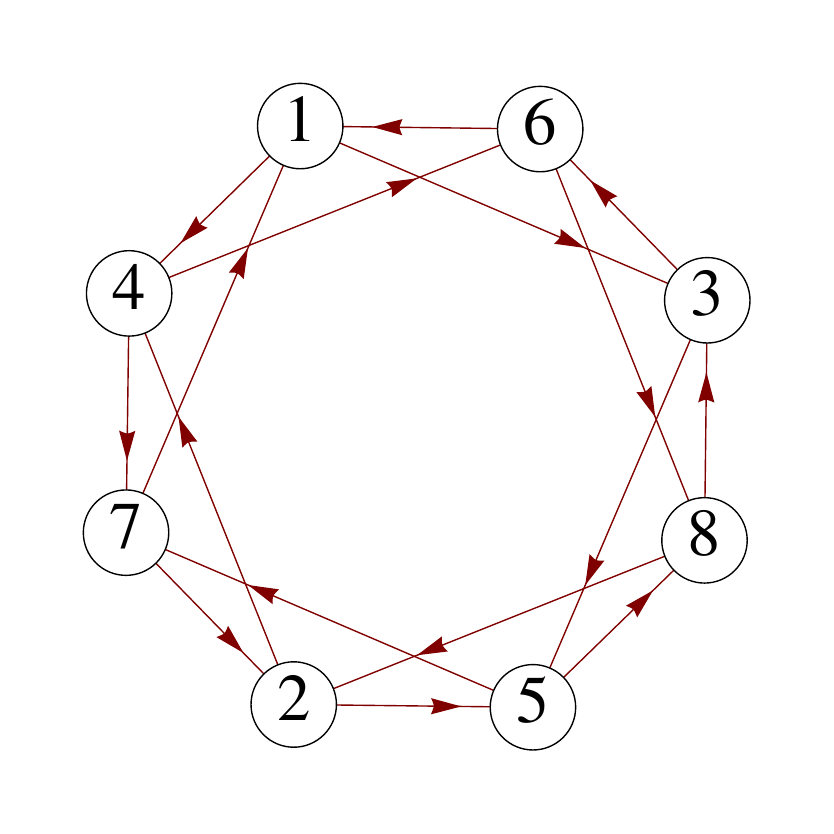}
\includegraphics[width=0.3\textwidth]{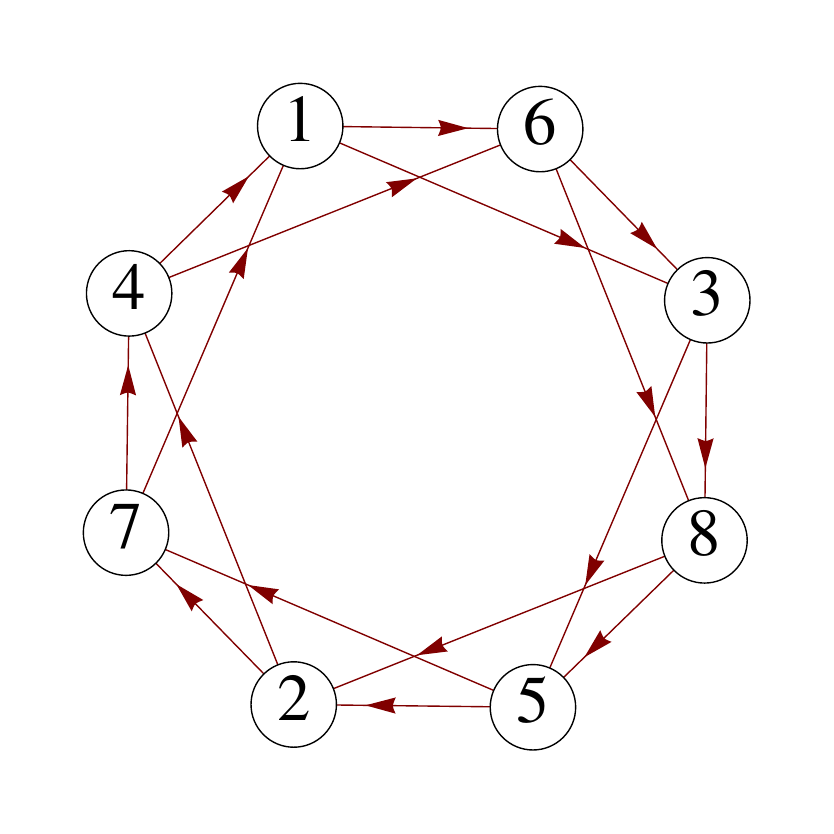}
\end{center}
\caption{The 4-regular circulant graph $C_{8}(2,3)$ and the circulant digraphs $C_{8}^+(2,3)$, $C_{8}^+(2,5)$. The corresponding diameters are 2, 3 and 4, respectively. \label{figCG8}}
\end{figure}

\begin{figure}
\begin{center}
\includegraphics[width=0.3\textwidth]{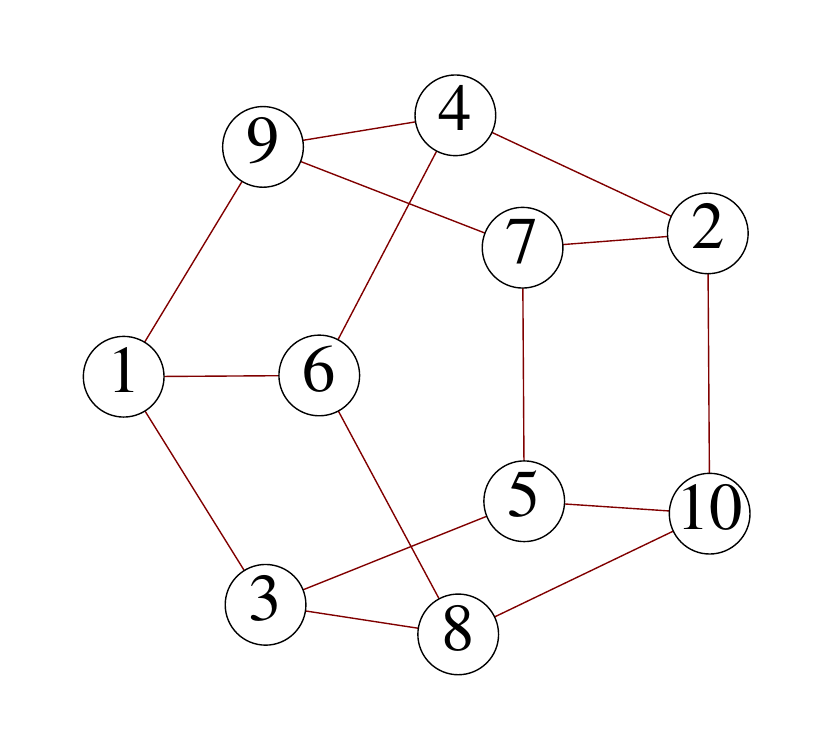}
\includegraphics[width=0.3\textwidth]{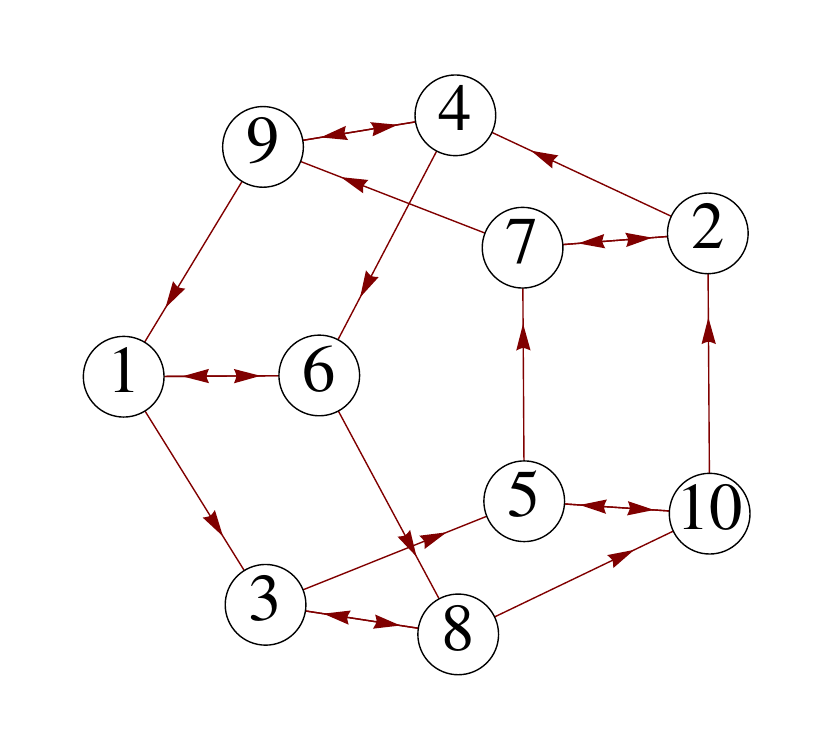}
\includegraphics[width=0.3\textwidth]{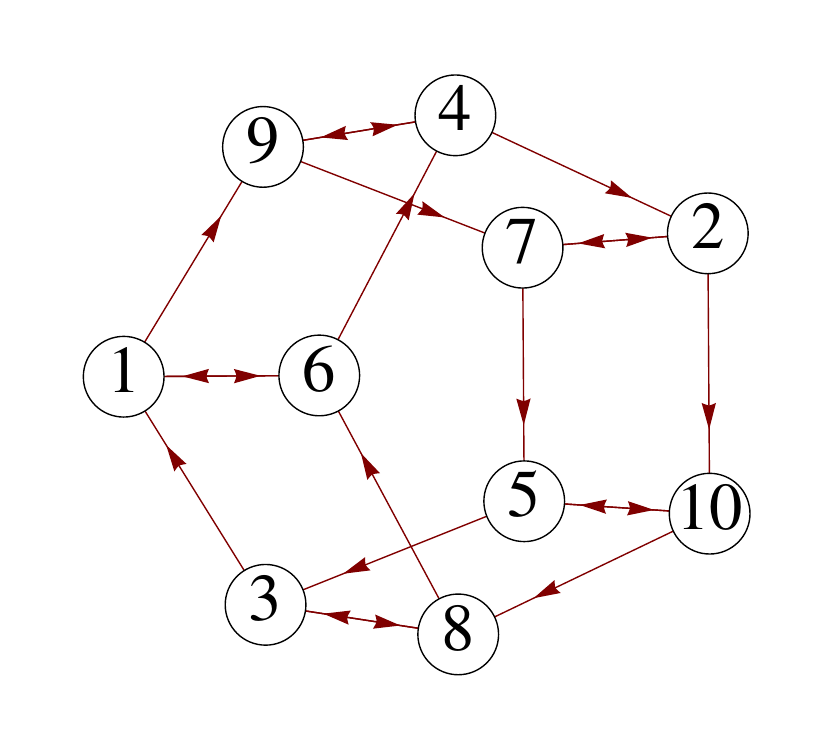}
\end{center}
\caption{The 3-regular circulant graph $C_{10}(2,5)$ and the circulant digraphs $C_{10}^+(2,5)$, $C_{10}^+(5,8)$. The corresponding diameters are 3, 5 and 5, respectively. \label{figCG10}}
\end{figure}

To construct a {\em directed} circulant graph ({\em circulant digraph} for short) choose an integer vector $\veca=(a_1,\ldots,a_k)$ with distinct positive coefficients $0<a_1<\ldots<a_k<n$. The circulant digraph $C_n^+(\veca)$ is defined to have an edge from $i$ to $j$ whenever $j-i \equiv a_h \bmod n$ for some $h\in\{ 1,\ldots, k\}$. In $C_n^+(\veca)$, every vertex has precisely $k$ outgoing and $k$ incoming edges.
$C_n^+(\veca)$ is strongly connected if and only if $\gcd(a_1,\ldots,a_k,n)=1$. In this case $C_n^+(\veca)$ is the directed Cayley graph of $\Z/n\Z$ with respect to the generating set $\{a_1,\ldots,a_k\}$.

Fix a vector $\vecell=(\ell_1,\ldots,\ell_k)\in\RR_{>0}^k$. We endow our circulant (di-)graph with a \mbox{(quasi-)metric} by stipulating that the edge from $i$ to $j\equiv i+a_h \bmod n$ has length $\ell_h$. We denote the corresponding metric graphs by $C_n(\vecell,\veca)$ and $C_n^+(\vecell,\veca)$, respectively. The distance $d(i,j)$ between two vertices is the length of the shortest path from $i$ to $j$. The diameter is the maximal distance between any pair of vertices,
\begin{equation}
	\diam=\max_{i,j} d(i,j) .
\end{equation}

To define an ensemble of random circulant graphs, we set
\begin{align}
&\fF^+:=\{ \vecx\in\RR^{k+1} : 0<x_1<\ldots<x_k<x_{k+1} \};
\qquad
\fF:=\fF^+\cap\{x_k\leq \sfrac12 x_{k+1} \},
\end{align}
and then in the directed case we fix an arbitrary bounded subset 
$\scrD\subset\fF^+$ with nonempty interior and boundary of Lebesgue measure
zero;
in the undirected case we fix an arbitrary bounded subset $\scrD\subset\fF$
subject to the same conditions.
Denote by $\widehat\NN^{k+1}$ the set of integer vectors in $\RR^{k+1}$ with positive coprime coefficients (i.e., the greatest common divisor of all coefficients is one). The numbers $(\veca,n)\in\widehat\NN^{k+1}$ defining $C_n(\vecell,\veca)$ or $C_n^+(\vecell,\veca)$ are then picked uniformly at random from 
the dilated set $T\scrD$ ($T>0$).
Note here that $\widehat\NN^{k+1}\cap T\scrD$ is nonempty for all large $T$;
in fact
\begin{align}\label{TDCARDINALITY}
	\#\big\{ (\veca,n)\in \widehat\NN^{k+1} \cap T\scrD \big\} \sim \frac{\vol(\scrD)}{\zeta(k+1)}\, T^{k+1},
\qquad\text{as }\: T\to\infty.
\end{align}

Our first main theorem shows that the (properly scaled) 
diameter of a random circulant digraph
has a limit distribution which is independent of the choice of $\scrD$.
In order to describe this limit distribution, we introduce some further
notation.
For a given closed bounded convex set $K$ of nonzero volume in $\R^k$ and a 
($k$-dimensional) lattice $L\subset\R^k$, 
we denote by $\rho(K,L)$ the covering radius
of $K$ with respect to $L$, i.e.\ the smallest
positive real number $r$ such that the translates of $rK$ by the vectors of $L$
cover all of $\R^k$:
\begin{align}
\rho(K,L)=\inf\bigl\{r>0\col rK+L=\R^k\bigr\}.
\end{align}
Let $X_k$ be the set of all lattices $L\subset\R^k$ of covolume one,
and let $\mu_0$ be the $\SL(k,\R)$ invariant probability measure on $X_k$.
Also let $\Delta$ be the simplex
\begin{equation}\label{simplex}
	\Delta= \big\{ \vecx\in\RR_{\geq 0}^k : x_1+\ldots+x_k \leq 1 \big\} .
\end{equation}

\begin{thm}\label{Thm1}
Let $k\geq2$. Then for any $\vecell\in\RR_{>0}^k$ and any bounded set 
$\scrD\subset\fF^+$ with nonempty interior and 
boundary of Lebesgue measure zero,
we have convergence in distribution
\begin{align}\label{Thm1res}
\frac{ \diam C_n^+(\vecell,\veca)}{(n \ell_1\cdots \ell_k)^{1/k}}
\xrightarrow[]{\textup{ d }}\rho(\Delta,L)
\qquad\text{as }\: T\to\infty,
\end{align}
where the random variable in the left-hand side is
defined by taking $(\veca,n)$ uniformly at random in 
$\widehat\NN^{k+1}\cap T\scrD$, and the random variable in the right-hand side
is defined by taking $L$ at random in $X_k$ according to $\mu_0$.
\end{thm}

\begin{figure}
\begin{center}
\includegraphics[width=0.7\textwidth]{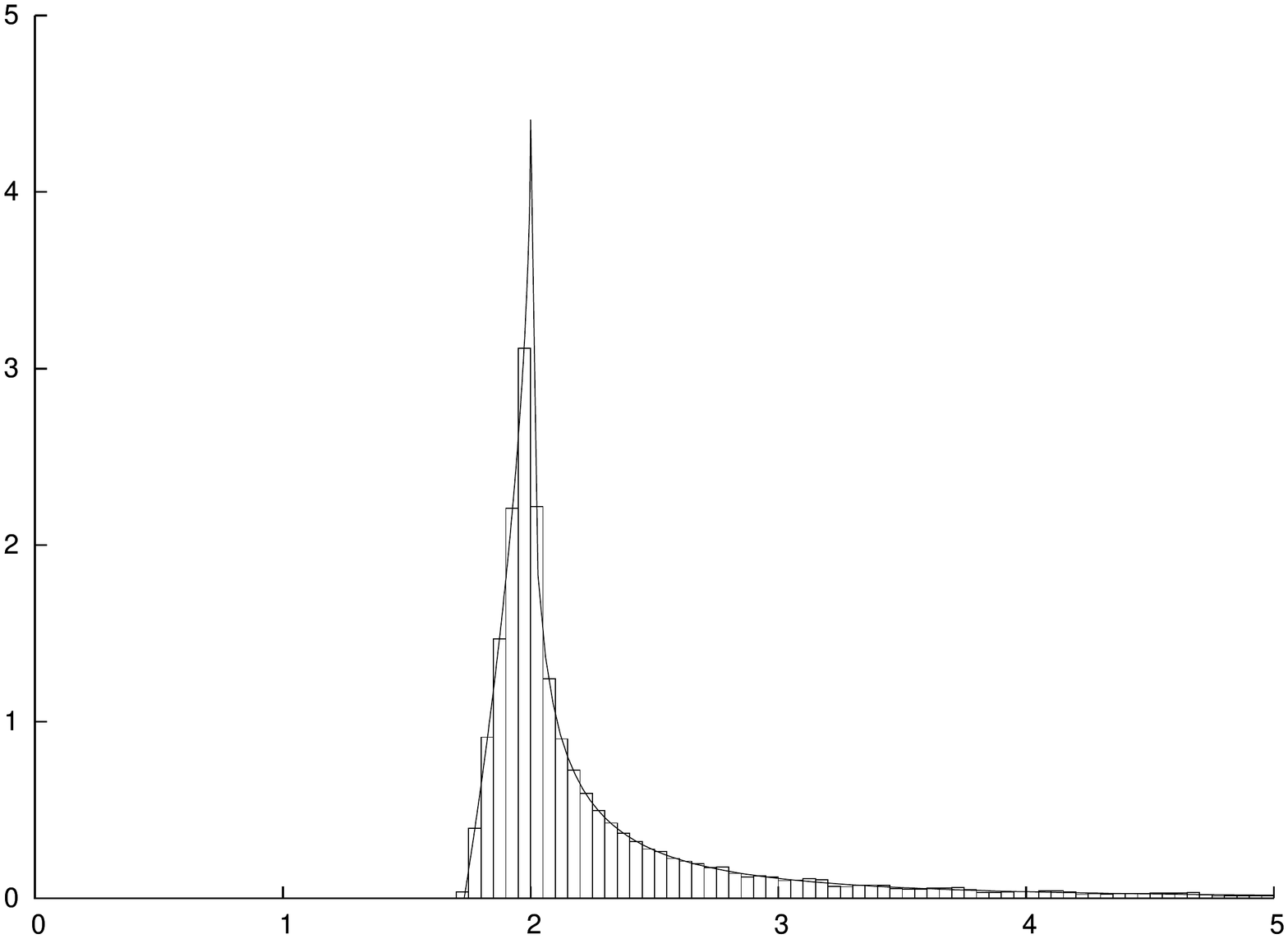}
\end{center}
\vspace{-30pt}
\caption{Distribution of diameters $n^{-1/k}(\diam C_n^+(\vecell,\veca)+\vece\cdot\vecell)$ for circulant digraphs with $k=2$  and $\vecell=\vece:=(1,1)$ vs.\ Ustinov's distribution $p_2(R)$ in \eqref{Ustinovs}. The numerical computations assume $\scrD=\fF^+\cap\{x_3\leq1\}$ and $T=1000$. \label{figlimit-dir}}
\end{figure}

\begin{figure}
\begin{center}
\includegraphics[width=0.8\textwidth]{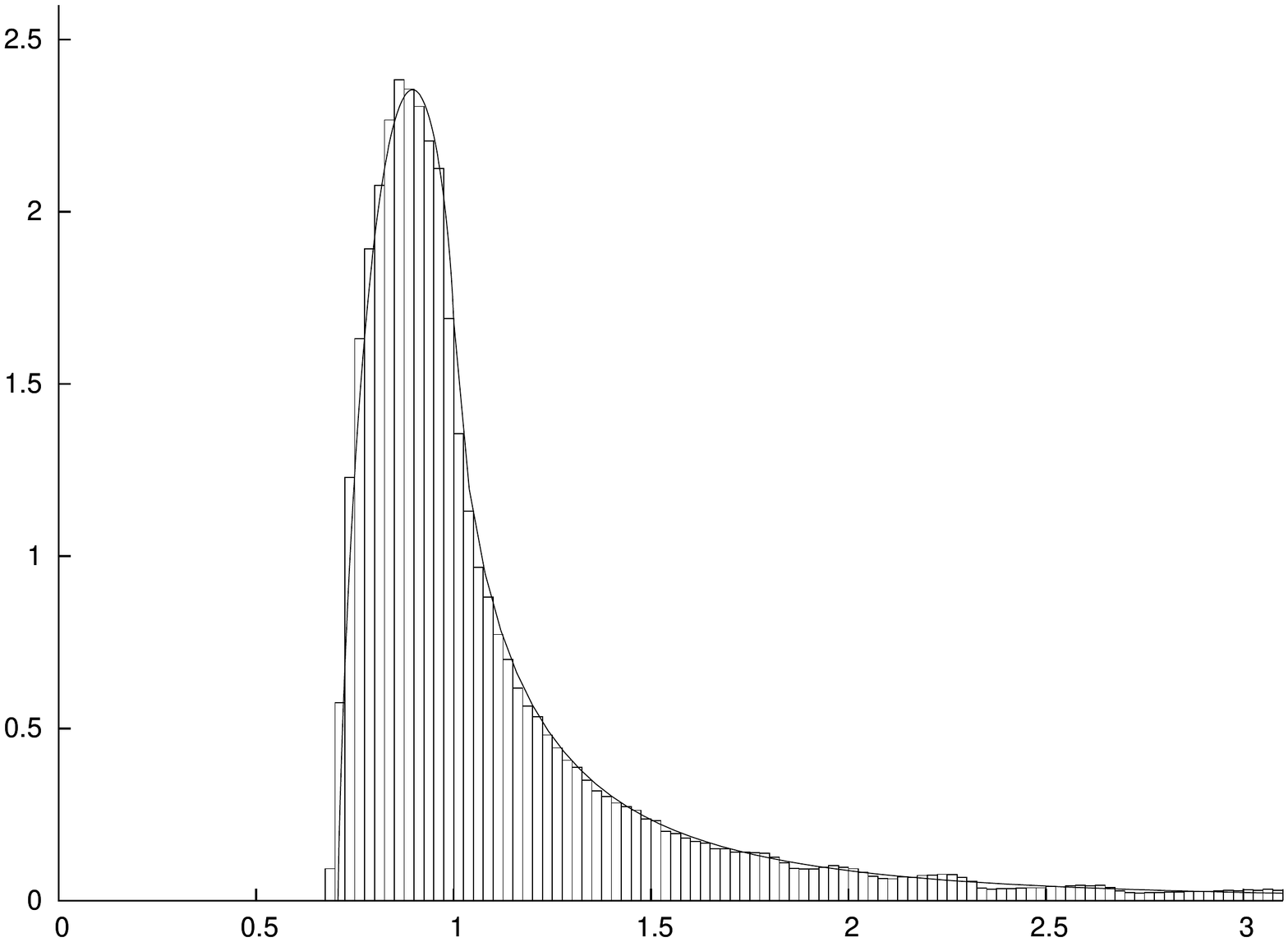}   %
\end{center}
\vspace{-30pt}
\caption{Distribution of diameters $n^{-1/k}\diam C_n(\vece,\veca)$ for circulant graphs with $k=2$ vs.\ our formula \eqref{expli}. The numerical computations assume $\scrD=\fF\cap\{x_3\leq1\}$ and $T=1000$. \label{figlimit}}
\end{figure}

\begin{remark}\label{PKEXPLREM}
The limit distribution in Theorem \ref{Thm1} is the same as the limit distribution for Frobenius numbers in $d=k+1$ variables found in \cite{Marklof10}, 
and our proof depends crucially on the equidistribution result 
proved in \cite[Thms.\ 6, 7]{Marklof10}.
Let $P_k(R)$ be the complementary distribution function of 
$\rho(\Delta,L)$, viz.
\begin{align}\label{PKRDEF}
P_k(R):=\mu_0\bigl(\bigl\{L\in X_k\col \rho(\Delta,L)>R\bigr\}\bigr).
\end{align}
($P_k(R)=\Psi_d(R)$ in the notation of \cite{Marklof10}.)
It was proved in \cite{Marklof10} that $P_k(R)$ is continuous for any fixed
$k\geq2$.
Hence, recalling also \eqref{TDCARDINALITY}, the statement of
Theorem \ref{Thm1} is equivalent 
with the statement that for any $R\geq0$ we have
\begin{align}\label{Thm1eq}
	\lim_{T\to\infty} \frac{1}{T^{k+1}} \#\bigg\{ (\veca,n)\in \widehat\NN^{k+1} \cap T\scrD\col  \frac{ \diam C_n^+(\vecell,\veca)}{(n \ell_1\cdots \ell_k)^{1/k} } > R \bigg\} = \frac{\vol(\scrD)}{\zeta(k+1)}\,P_k(R).
\end{align}
 \end{remark}

We also remark that Li \cite{hL2010} has recently proved effective 
versions of the equidistribution results in \cite{Marklof10}.
Using Li's work it should be possible to also prove effective versions of
our Theorems \ref{Thm1}, \ref{Thm2}, as well as
Theorems \ref{equiThm3}, \ref{equiThm2} in Section \ref{MAINPROOFSEC}.

\begin{remark}\label{lowbo}
In analogy with the case of Frobenius numbers \cite{Aliev07}, we also obtain the following sharp lower bound, writing $\vece:=(1,\ldots,1)\in\R^k$,
\begin{equation}
	\frac{ \diam C_n^+(\vecell,\veca) + \vece\cdot\vecell}{(n \ell_1\cdots \ell_k)^{1/k} } \geq \rho_k,\qquad \text{with }\:
	\rho_k := \inf_{L\in X_k} \rho(\Delta,L) .
\end{equation}
It follows from the description in Remark \ref{PKEXPLREM} that
\begin{align}\label{PKSUPPORT}
P_k(R)=1\:\text{ for }\: 0\leq R\leq\rho_k,
\qquad\text{and}\qquad
0<P_k(R)<1\:\text{ for }\: R>\rho_k.
\end{align}
It is proved in \cite{Aliev07} that $\rho_k>(k!)^{1/k}$, and in fact
for $k$ large,
$\rho_k$ is not much larger than %
$(k!)^{1/k}$; 
indeed
$\rho_k\leq (k!)^{1/k}(1+O(k^{-1}\log k))$ %
(cf.\ \cite[Sec.\ 9]{rDvF2004}, \cite{pG85}, \cite{cR59}).
Also for $k$ large, the limit distribution described by $P_k(R)$ has
almost all of its mass concentrated between
$(k!)^{1/k}$ and $1.757\cdot(k!)^{1/k}$.
In fact, for any fixed $\alpha>1+\eta_0$, where
$\eta_0=0.756\ldots$ is the unique real root of 
$e\log \eta+\eta=0$, $P_k(\alpha(k!)^{1/k})$ tends to zero with an
exponential rate as $k\to\infty$ \cite[Thm.\ 4.1]{strombergsson11}.
\end{remark}

\begin{remark}
For $k$ fixed and $R$ large,
\begin{equation}\label{PKASYMPT}
	P_k(R)=\frac{k+1}{2\zeta(k)}R^{-k}+O_k\bigl(R^{-k-1-\frac1{k-1}}\bigr).
\end{equation}
This asymptotic formula is proved in \cite[Thm.\ 1.2]{strombergsson11}.
The upper bound $P_k(R)\ll R^{-k}$ had previously been proved in \cite{hL2010}.
\end{remark}

\begin{remark}
For $k=2$, Theorem \ref{Thm1} has been proved by Ustinov by different methods, see the last section of \cite{Ustinov10}. This paper also computes an explicit formula for the limit density $p_k(R)=-\frac{d}{dR}P_k(R)$ (which coincides with the distribution of Frobenius numbers for three variables):
\begin{equation}\label{Ustinovs}
p_2(R)=
\begin{cases}
0 & (0\leq R \leq \sqrt 3)  \\
\frac{12}{\pi}\big(\frac{R}{\sqrt 3}-\sqrt{4-R^2}\big) & (\sqrt 3\leq R \leq 2)  \\
\frac{12}{\pi^2}\big( R\sqrt 3 \arccos\big(\frac{R+3\sqrt{R^2-4}}{4\sqrt{R^2-3}}\big)+\frac32 \sqrt{R^2-4}  \log\big(\frac{R^2-4}{R^2-3}\big)\big) & (R>2).
\end{cases}
\end{equation}
We give an alternative proof of this formula, deriving it as a 
consequence of \eqref{PKRDEF}, in Section \ref{p2REXPLSEC} below.
\end{remark}

We now turn to the case of {\em undirected} circulant graphs. The following theorem says in particular that, as in the directed case, the limit distribution
for the diameter is independent of the choice of $\scrD$.
Let $\fP$ be the (regular) polytope
\begin{equation}\label{poly}
	\fP= \big\{ \vecx\in\RR^k : |x_1|+\ldots+|x_k|\leq 1 \big\}.
\end{equation}
This %
is a $k$-dimensional cross-polytope, 
cf.\ \cite{hC73};
in particular $\fP$ is a square for $k=2$ and an octahedron for $k=3$.

\begin{thm}\label{Thm2}
Let $k\geq2$.
Then for any $\vecell\in\RR_{>0}^k$ and any bounded set 
$\scrD\subset\fF$ with nonempty interior and 
boundary of Lebesgue measure zero,
we have convergence in distribution
\begin{align}\label{Thm2res}
\frac{ \diam C_n(\vecell,\veca)}{(n \ell_1\cdots \ell_k)^{1/k}}
\xrightarrow[]{\textup{ d }}\rho(\fP,L)
\qquad\text{as }\: T\to\infty,
\end{align}
where the random variable in the left-hand side is
defined by taking $(\veca,n)$ uniformly at random in 
$\widehat\NN^{k+1}\cap T\scrD$, and the random variable in the right-hand side
is defined by taking $L$ at random in $X_k$ according to $\mu_0$.
\end{thm}

\begin{remark}
Let $\tilde P_k(R)$ be the complementary distribution function of 
$\rho(\fP,L)$, viz.
\begin{align}\label{TILDEPKDEF}
\tilde P_k(R):=\mu_0\bigl(\bigl\{L\in X_k\col \rho(\fP,L)>R\bigr\}\bigr).
\end{align}
This function is continuous (cf.\ %
Section \ref{TILDEPCONTSEC}
below), and hence,
recalling also \eqref{TDCARDINALITY}, the statement of 
Theorem \ref{Thm2} is equivalent with the statement that for any $R\geq0$ 
we have
\begin{equation}\label{Thm2eq}
	\lim_{T\to\infty} \frac{1}{T^{k+1}} \#\bigg\{ (\veca,n)\in \widehat\NN^{k+1} \cap T\scrD \col  \frac{ \diam C_n(\vecell,\veca)}{(n \ell_1\cdots \ell_k)^{1/k} } > R \bigg\} = \frac{\vol(\scrD)}{\zeta(k+1)}\, \tilde P_k(R) .
\end{equation}
\end{remark}

\begin{remark}\label{lowbo2}
We have the lower bound (cf.\ Proposition \ref{DIAMFINALRELPROP} and
Lemma \ref{DIAMASCOVRADLEM} below)
\begin{equation}\label{TILDERHOKDEF}
	\frac{ \diam C_n(\vecell,\veca) + \frac12 \vece\cdot\vecell}{(n \ell_1\cdots \ell_k)^{1/k} } \geq \tilde\rho_k,\qquad \text{with }\:
	\tilde\rho_k := \inf_{L\in X_k}\rho(\fP,L).
\end{equation}
(Recall $\vece:=(1,\ldots,1)\in\R^k$).
Also the distribution described by $\tilde P_k(R)$ has support exactly in
the interval $[\tilde\rho_k,\infty)$,
in analogy with \eqref{PKSUPPORT}.
Since any covering of $\R^k$ has density at least one we have 
\begin{align}\label{TILDERHOKLOWBOUND}
\tilde\rho_k\geq\vol(\fP)^{-1/k}=\sfrac12(k!)^{1/k}.
\end{align}
In fact \eqref{TILDERHOKLOWBOUND} holds with equality for $k=2$;
$\tilde\rho_2=\frac1{\sqrt2}$,
since there exist lattice coverings of $\R^2$ by squares without
any overlap;
however for every $k\geq3$ we have strict inequality in
\eqref{TILDERHOKLOWBOUND}; cf.\ Section \ref{TILDERHOSTRICTINEQSEC} below.
We also have 
$\tilde\rho_k\leq\frac12(k!)^{1/k}(1+O(k^{-1}\log k))$
(again cf.\ \cite[Sec.~9]{rDvF2004}, \cite{pG85}, \cite{cR59}),
and for $k$ large, the limit distribution described by $\tilde P_k(R)$
has almost all of its mass concentrated between
$\frac12(k!)^{1/k}$ and $1.757\cdot\frac12(k!)^{1/k}$,
in the same sense as for $P_k(R)$
\cite[Thm.~4.1]{strombergsson11}.
\end{remark}

\begin{remark}\label{PTILDEASYMPTREM}
For $k$ fixed and $R$ large,
we will show in Section \ref{PTILDEASYMPTSEC} that
\begin{equation}\label{PTILDEASYMPTREMRES}
\tilde P_k(R)=\frac{R^{-k}}{2\zeta(k)}+O_k\bigl(R^{-k-1-\frac1{k-1}}\bigr).
\end{equation}
\end{remark}

\begin{remark}\label{TILDEPEXPLREM}
In the case $k=2$, the limit density $\tilde p_k(R)=-\frac{d}{dR}\tilde P_k(R)$ can be calculated explicity; we will show in Section \ref{secExplicit} that
\begin{equation}\label{expli}
\tilde p_2(R)=
\begin{cases} 0 & (0\leq R\leq\frac{1}{\sqrt 2}) \\
\frac{24}{\pi^2} \big(\frac{2R^2-1}{R}\log\big(\frac{2R^2}{2R^2-1}\big)+\frac{1-R^2}{R}\log\big(\frac{R^2}{|1-R^2|}\big)\big) & (R>\frac{1}{\sqrt 2}) .
\end{cases}
\end{equation}
\end{remark}

The outline of the paper is as follows.
In Section \ref{MAINPROOFSEC} we prove Theorems \ref{Thm1} and \ref{Thm2},
by realizing the
circulant graphs as lattice graphs on flat tori, and applying the central
equidistribution result proved in \cite{Marklof10}.
In Section \ref{RHOPLDISTRSEC} we prove the assertions which we 
have made about the limit distribution in Theorem \ref{Thm2},
viz.\ that the distribution function $R\mapsto\tilde P_k(R)$
is continuous, that we have strict inequality
$\tilde\rho_k>\frac12(k!)^{1/k}$ for every $k\geq3$,
and that $\tilde P_k(R)$ has the precise polynomial decay as given by
\eqref{PTILDEASYMPTREMRES}.
In Section \ref{secExplicit} we prove the explicit formula for 
$\tilde p_2(R)$, and also give a new proof of the explicit formula
for $p_2(R)$.
Finally in Section \ref{secEx} we discuss a number of natural
extensions and variations of Theorems \ref{Thm1} and \ref{Thm2}.

\subsection*{Acknowledgements}
We are grateful to Svante Janson for inspiring and helpful discussions.

\begin{figure}
\begin{center}
\includegraphics[width=0.7\textwidth]{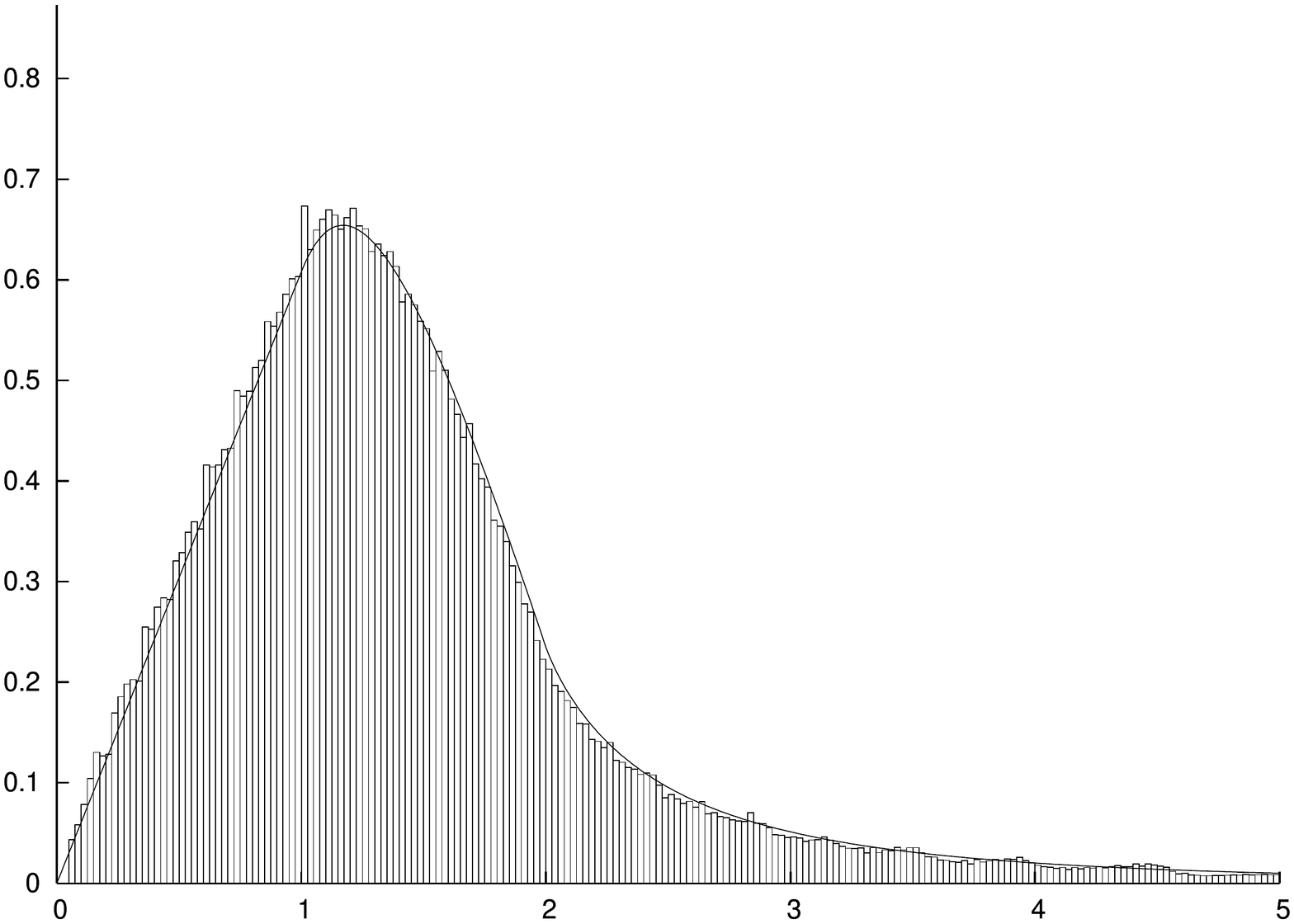}
\end{center}
\vspace{-30pt}
\caption{Distribution of the shortest cycle length $n^{-1/k}\scl C_n^+(\vece,\veca)$ for circulant digraphs with $k=2$ vs.\ the probability density $p_{2,\scl}(R)$ discussed in Section \ref{secEx}. The numerical computations assume $\scrD=\fF^+\cap\{x_3\leq1\}$ and $T=1000$. \label{figcycle-dir}}
\end{figure}

\begin{figure}
\begin{center}
\includegraphics[width=0.7\textwidth]{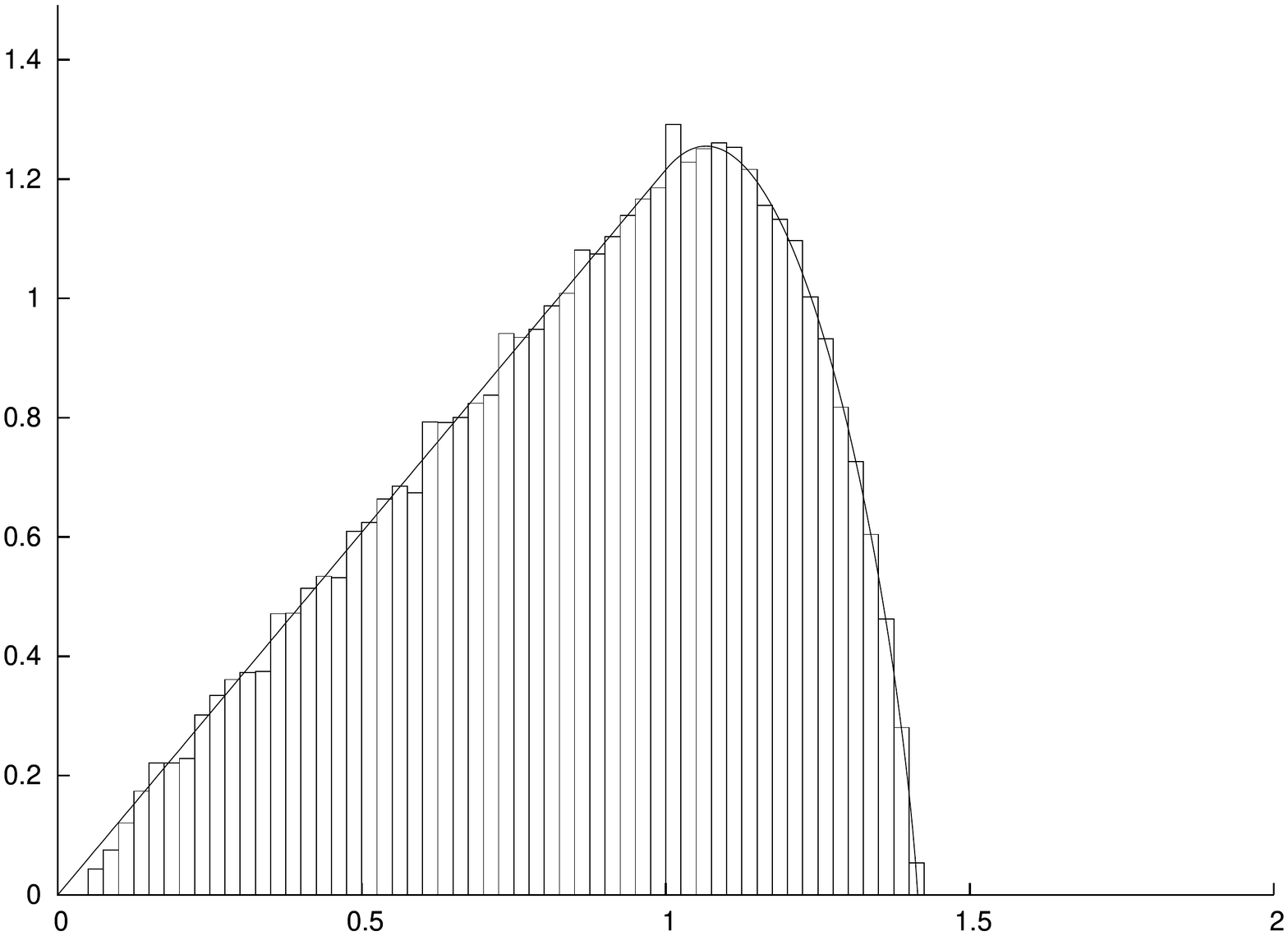}
\end{center}
\vspace{-30pt}
\caption{Distribution of the shortest non-trivial cycle length $n^{-1/k}\scl C_n(\vece,\veca)$ for circulant graphs with $k=2$ vs.\ the probability density $\tilde p_{2,\scl}(R)$ discussed in Section \ref{secEx}. The numerical computations assume $\scrD=\fF\cap\{x_3\leq1\}$ and $T=1000$. \label{figcycle}}
\end{figure}

\section{Lattice graphs on flat tori and their continuum limit}
\label{MAINPROOFSEC}

In this section we will prove Theorems \ref{Thm1} and \ref{Thm2}. 
The first step is to
realize an arbitrary circulant graph as a lattice graph on a flat torus.
This has previously been %
used in \cite{rDvF2004} and \cite{sCjSmAtC2010}; we here give
an alternative presentation, adapted so as to make the equidistribution
results from \cite{Marklof10} apply in a transparent fashion.

\subsection{Directed lattice graphs}

Let $LG_{k}^+=(\ZZ^{k},E)$ be the standard directed lattice graph with vertex set $\ZZ^{k}$; the edge set $E$ comprises all directed edges of the form $(\vecm,\vecm+\vece_h)$ where $\vecm\in\ZZ^k$ and $\vece_1,\ldots,\vece_k$ is the standard basis. We define a quasimetric on $LG_k^+$ by fixing $\vecell=(\ell_1,\ldots,\ell_k)\in\RR_{>0}^k$ and assigning length $\ell_h$ to every edge of the form $(\vecm,\vecm+\vece_h)$. The distance from vertex $\vecm$ to $\vecn$ in $LG_{k}^+$ is then given by
\begin{equation}
d(\vecm,\vecn)=
\begin{cases}
(\vecn-\vecm)\cdot\vecell & \text{if $\vecn-\vecm\in \ZZ_{\geq 0}^k$,}\\
\infty & \text{otherwise.}
\end{cases}
\end{equation}

If $\Lambda$ is a sublattice of $\ZZ^k$ we define the quotient lattice graph 
$LG_{k}^+/\Lambda$ as the digraph with vertex set $\ZZ^k/\Lambda$ and edge set 
\begin{align}
\bigl\{(\vecm+\Lambda,\vecm+\vece_h+\Lambda):\vecm\in\ZZ^k,\: h=1,\ldots,k\bigr\}.
\end{align}
(Note that edges of the form $(\vecm+\Lambda,\vecm+\Lambda)$ correspond to loops.)
The distance from vertex $\vecm+\Lambda$ to $\vecn+\Lambda$ in $LG_k^+/\Lambda$ is
\begin{align}\label{LGPQUOTDISTDEF}
d\bigl(\vecm+\Lambda,\vecn+\Lambda\bigr)=\begin{cases}
\min\bigl((\vecn-\vecm+\Lambda)\cap\ZZ_{\geq0}^k\bigr)\cdot\vecell
&\text{if }\:(\vecn-\vecm+\Lambda)\cap\ZZ_{\geq0}^k\neq\emptyset,
\\
\infty&\text{otherwise.}\end{cases}
\end{align}

Set $d=k+1$.
Given $(\veca,n)=(a_1,\ldots,a_k,n)\in\hatNN^d$ with $0<a_1<\cdots<a_{k}<n$,
we introduce the following sublattices of $\Z^d$:
\begin{equation}\label{KNADEF}
\Lambda_n=\ZZ^k\times n\ZZ \quad\text{and}\quad \Lambda_n(\veca)= \Lambda_n u(\veca),
\end{equation}
where
\begin{equation}\label{UADEF}
u(\veca):=\begin{pmatrix} 1_k & \trans\veca \\ \vecnull & 1 \end{pmatrix}
\in\SL(d,\Z).
\end{equation}
For a subset $Y\subset\RR^d$ we denote by $Y_0$ the set $Y\cap (\RR^k\times\{0\})$; we view $Y_0$ as a subset of $\RR^k$.

\begin{lem}\label{DIGRAPHISOLEM}
The set $\Lambda_n(\veca)_0$ is a sublattice of $\ZZ^k$ of index $n$; furthermore the quasimetric digraphs $LG_{k}^+/\Lambda_n(\veca)_0$ and $C_n^+(\vecell,\veca)$ are isomorphic.
\end{lem}
\begin{proof}
An integer vector $\vecm\in\Z^k$ lies in
$\Lambda_n(\veca)_0$ if and only if $(\vecm,0)\in \Lambda_n(\veca)$,
and this holds if and only if $\vecm\cdot\veca\equiv0\minmod n$.
In other words $\Lambda_n(\veca)_0$ is the kernel of the homomorphism 
$\vecm\mapsto\vecm\cdot\veca\minmod n$ from $\Z^k$ onto $\Z/n\Z$.
Hence $\Lambda_n(\veca)_0$ is indeed a sublattice of $\Z^k$ of index $n$,
and %
the map just considered induces an isomorphism
$J:\Z^k/\Lambda_n(\veca)_0\overset\sim\to\Z/n\Z$.
Note that $J(\vece_h+\Lambda_n(\veca)_0)=a_h\minmod n$;
hence the edge set of $LG_k^+/\Lambda_n(\veca)_0$ is
\begin{align}
\bigl\{(J^{-1}(j),J^{-1}(j+a_h))\col j\in\Z/n\Z,\: h=1,\ldots,k\bigr\},
\end{align}
where the length of any edge $(J^{-1}(j),J^{-1}(j+a_h))$ is $\ell_h$.
Hence $J$ yields an isomorphism between the digraphs
$LG_{k}^+/\Lambda_n(\veca)_0$ and $C_n^+(\vecell,\veca)$,
preserving the quasimetric.
\end{proof}

\subsection{Undirected lattice graphs}

The discussion of the previous section applies with very small changes to the undirected lattice graph $LG_{k}=(\ZZ^{k},E)$, where the edge set $E$ is the same as before but the edges are considered without orientation. 

The metric on $LG_k$ is defined as for $LG_k^+$, and now the distance 
between vertices $\vecm,\vecn\in LG_k$ is given by
\begin{align}
d(\vecm,\vecn)=(\vecn-\vecm)_+\cdot\vecell,
\end{align}
where we denote $\vecz_+:=(|z_1|,\ldots,|z_k|)$ for any 
$\vecz=(z_1,\ldots,z_k)\in\R^k$.
Furthermore if $\Lambda$ is a sublattice of $\Z^k$ then the
distance between vertices $\vecm+\Lambda$ and $\vecn+\Lambda$ in $LG_k/\Lambda$ is given by
\begin{align}\label{LGQUOTDISTDEF}
d\bigl(\vecm+\Lambda,\vecn+\Lambda\bigr)=\min\bigl\{\vecz_+\cdot\vecell\col
\vecz\in\vecn-\vecm+\Lambda\bigr\}.
\end{align}
Now take $(\veca,n)=(a_1,\ldots,a_k,n)\in\widehat\NN^d$ with
$0<a_1<\ldots<a_k\leq \frac{n}{2}$, and recall the definitions
\eqref{KNADEF} and \eqref{UADEF}.
\begin{lem}\label{GRAPHISOLEM}
The metric graphs $LG_k/\Lambda_n(\veca)_0$ and $C_n(\vecell,\veca)$ are 
isomorphic.
\end{lem}
The proof is the same as for Lemma \ref{DIGRAPHISOLEM}.

\subsection{Diameters}

Let $\Lambda$ be a sublattice of $\Z^k$ of full rank (viz., of finite index).
In view of the definition of the distance on $LG_{k}/\Lambda$, we have for the diameter
\begin{align}\label{DIAMLG}
\diam(LG_k/\Lambda)=
\max_{\vecm\in\Z^k/\Lambda}\min\bigl\{\vecz_+\cdot\vecell\col\vecz\in\vecm+\Lambda\bigr\}.
\end{align}
We define a corresponding diameter for the continuous torus $\R^k/\Lambda$:
\begin{align}\label{DIAMRKK}
\diam_\vecell(\R^k/\Lambda)=
\sup_{\vecy\in\R^k/\Lambda}\min\bigl\{\vecz_+\cdot\vecell\col\vecz\in\vecy+\Lambda\bigr\}.
\end{align}
This is the maximal distance between any two points on $\R^k/\Lambda$,
when distance is measured in the ``$\vecell$-weighted $\ell^1$-metric'',
i.e.\ we define the distance between any two points
$\vecx+\Lambda$ and $\vecy+\Lambda$ on $\R^k/\Lambda$ as the minimum of
$\vecz_+\cdot\vecell$ taken over all $\vecz\in\vecy-\vecx+\Lambda$.

Similarly for the directed graph $LG_{k}^+/\Lambda$ we have
\begin{equation}\label{DIAMLGP}
\diam(LG_{k}^+/\Lambda)=\max_{\vecm\in\ZZ^k/\Lambda} 
\min \big((\vecm+\Lambda)\cap \ZZ_{\geq 0}^k \big)\cdot\vecell .
\end{equation}
We define a corresponding directed diameter for the continuous torus $\RR^k/\Lambda$:
\begin{equation}\label{DIAMRKKP}
\diam_\vecell^+(\RR^k/\Lambda)=\sup_{\vecy\in\RR^k/\Lambda}
\min \big((\vecy+\Lambda)\cap \RR_{\geq 0}^k \big)\cdot\vecell .
\end{equation}
This is the maximal distance between any two points on $\R^k/\Lambda$,
when distance is measured in the $\vecell$-weighted $\ell^1$-metric,
and we only allow paths with non-negative components.

Recall that we write $\vece=(1,\ldots,1)\in\R^k$.

\begin{lem}\label{DIAMDISCRCONTLEM}
Let $(\veca,n)=(a_1,\ldots,a_k,n)\in\hatNN^d$ with $0<a_1<\cdots<a_{k}<n$.
Then
\begin{equation}\label{DIAMDISCRCONTLEMRES1}
\diam\bigl(LG_{k}^+/\Lambda_n(\veca)_0\bigr)
=\diam_\vecell^+\bigl(\RR^k/\Lambda_n(\veca)_0\bigr)-\vece\cdot\vecell.
\end{equation}
If furthermore $a_k\leq\frac n2$ then
\begin{equation}\label{DIAMDISCRCONTLEMRES2}
\diam_\vecell\bigl(\RR^k/\Lambda_n(\veca)_0\bigr)-\frac{\vece\cdot\vecell}{2}
\leq \diam\bigl(LG_{k}/\Lambda_n(\veca)_0\bigr)\leq 
\diam_\vecell\bigl(\RR^k/\Lambda_n(\veca)_0\bigr).
\end{equation}
\end{lem}

\begin{proof}
Set $\Lambda=\Lambda_n(\veca)_0$. 
Let $\vecy\in\R^k$ be arbitrary.
Set $\vecm:=(\lfloor y_1\rfloor,\ldots,\lfloor y_k\rfloor)\in\Z^k$,
so that $\vecy=\vecm+\vecz$ for some vector $\vecz\in[0,1)^k$.
Using $\Lambda\subset\Z^k$ we have
\begin{align}
(\vecy+\Lambda)\cap\R_{\geq0}^k=\vecz+((\vecm+\Lambda)\cap\Z_{\geq0}^k),
\end{align}
and thus
\begin{align}
\min\bigl((\vecy+\Lambda)\cap\R_{\geq0}^k\bigr)\cdot\vecell
=\vecz\cdot\vecell+\min\bigl((\vecm+\Lambda)\cap\Z_{\geq0}^k\bigr)\cdot\vecell.
\end{align}
Taking the supremum over all $\vecy\in\R^k$, or equivalently
the supremum over all $\langle\vecm,\vecz\rangle\in\Z^k\times[0,1)^k$,
we obtain
\begin{align}
\diam_\vecell^+(\R^k/\Lambda)
=\sup_{\vecz\in[0,1)^k}\vecz\cdot\vecell+\diam(LG_k^+/\Lambda)
=\vece\cdot\vecell+\diam(LG_k^+/\Lambda),
\end{align}
and we have proved \eqref{DIAMDISCRCONTLEMRES1}.

We next turn to \eqref{DIAMDISCRCONTLEMRES2}.
The right inequality in \eqref{DIAMDISCRCONTLEMRES2} is obvious from 
\eqref{DIAMLG} and \eqref{DIAMRKK}.
To prove the left inequality, let $\vecy=(y_1,\ldots,y_k)$ be an 
arbitrary point in $\R^k$. Then there is an integer vector
$\vecm=(m_1,\ldots,m_k)$ satisfying $|m_j-y_j|\leq\frac12$
for $j=1,\ldots,k$.
Now for any $\vecz\in\vecm+\Lambda$ there is a point
$\vecz'\in\vecy+\Lambda$ satisfying $|z_j'-z_j|\leq\frac12$ for all $j$.
Hence
\begin{align}
\min\bigl\{\vecz_+\cdot\vecell\col\vecz\in\vecy+\Lambda\bigr\}
\leq\min\bigl\{\vecz_+\cdot\vecell\col\vecz\in\vecm+\Lambda\bigr\}
+\frac{\vece\cdot\vecell}2
\leq\diam(LG_k/\Lambda)+\frac{\vece\cdot\vecell}2.
\end{align}
Since this holds for all $\vecy\in\R^k$ we obtain the left inequality
in \eqref{DIAMDISCRCONTLEMRES2}.
\end{proof}

Now set
\begin{equation}
D_n(\vecell)=
\diag\bigl(\Pi^{-1/k}\ell_1,\ldots,\Pi^{-1/k}\ell_k\bigr)\in\GL(k,\R),
\qquad\text{with }\:\Pi=n \ell_1\cdots\ell_k.
\end{equation}
We have $\det D_n(\vecell)=n^{-1}$, and hence
\begin{align}
L_{n,\veca,\vecell}:=\Lambda_n(\veca)_0 D_n(\vecell)
\end{align}
is a lattice in $\R^k$ of covolume one,
viz.\ $L_{n,\veca,\vecell}\in X_k$.
It is also clear from the definition \eqref{DIAMRKK} that this
transformation translates $\diam_\vecell(\R^k/\Lambda_n(\veca)_0)$
into an unweighted (or ``$\vece$-weighted'') $\ell^1$-diameter, viz.\ 
\begin{align}
\diam_\vecell(\R^k/\Lambda_n(\veca)_0)
=\Pi^{\frac1k}\diam_\vece\bigl(\R^k/L_{n,\veca,\vecell}
\bigr).
\end{align}
Similarly
\begin{align}
\diam_\vecell^+(\R^k/\Lambda_n(\veca)_0)
=\Pi^{\frac1k}\diam_\vece^+\bigl( \R^k/L_{n,\veca,\vecell}\bigr).
\end{align}

Combining Lemma \ref{DIGRAPHISOLEM}, Lemma \ref{GRAPHISOLEM}
and Lemma \ref{DIAMDISCRCONTLEM}, we have now proved:
\begin{prop}\label{DIAMFINALRELPROP}
Let $(\veca,n)=(a_1,\ldots,a_k,n)\in\hatNN^d$ with $0<a_1<\cdots<a_{k}<n$.
Then
\begin{equation}
\diam C_n^+(\vecell,\veca)
=\Pi^{\frac1k}\diam_\vece^+\bigl( \R^k/L_{n,\veca,\vecell}\bigr)
-\vece\cdot\vecell.
\end{equation}
If furthermore $a_k\leq\frac n2$ then
\begin{equation}
\Pi^{\frac1k}\diam_\vece\bigl(\R^k/L_{n,\veca,\vecell}\bigr)
-\frac{\vece\cdot\vecell}{2}
\leq\diam C_n(\vecell,\veca)
\leq
\Pi^{\frac1k}\diam_\vece\bigl(\R^k/L_{n,\veca,\vecell}\bigr).
\end{equation}
\end{prop}

\subsection{Diameters and covering radii}

We next note that, for an arbitrary $k$-dimensional 
lattice $\Lambda\subset\R^k$,
the $\ell^1$-diameters $\diam_\vece(\R^k/\Lambda)$ and 
$\diam_\vece^+(\R^k/\Lambda)$ can be interpreted as the covering radius
with respect to $\Lambda$ of the simplex $\Delta$ and the
cross-polytope $\fP$, respectively.
(Recall \eqref{simplex} and \eqref{poly}.)

\begin{lem}\label{DIAMASCOVRADLEM}
For any lattice $\Lambda\subset\R^k$ of full rank we have
\begin{align}\label{DIAMASCOVRADLEMRES1}
\diam_\vece(\R^k/\Lambda)=\rho(\fP,\Lambda)
\end{align}
and
\begin{align}\label{DIAMASCOVRADLEMRES2}
\diam_\vece^+(\R^k/\Lambda)=\rho(\Delta,\Lambda).
\end{align}
\end{lem}
\begin{proof}
Note that, for any $\vecy\in\R^k$,
\begin{align}
\min\bigl\{\vecz_+\cdot\vece\col\vecz\in\vecy+\Lambda\bigr\}
&=\sup\bigl\{R>0\col R\fP\cap(\vecy+\Lambda)=\emptyset\bigr\}.
\end{align}
Hence by \eqref{DIAMRKK}, $\diam_\vece(\R^k/\Lambda)$ equals the supremum 
of all $R>0$ such that there exists a translate of $\Lambda$ which is disjoint
from $R\fP$.
One sees that this holds if and only if $R\fP+\Lambda\neq\R^k$.
Hence
\begin{align}
\diam_\vece(\R^k/\Lambda)=\sup\bigl\{R>0\col R\fP+\Lambda\neq\R^k\bigr\}
=\rho(\fP,\Lambda).
\end{align}
The proof of \eqref{DIAMASCOVRADLEMRES2} is the same,
using the fact that
\begin{align}
\min\bigl((\vecy+\Lambda)\cap\R_{\geq0}^k\bigr)\cdot\vece
=\sup\bigl\{R>0\col R\Delta\cap(\vecy+\Lambda)=\emptyset\bigr\}
\end{align}
for all $\vecy\in\R^k$.
\end{proof}

\subsection{Equidistribution}

The key to the proof of Theorems \ref{Thm1} and \ref{Thm2} is the following 
equidistribution theorem, %
which is a consequence of Theorem 7 in \cite{Marklof10}. 

\begin{thm}\label{equiThm3}
Let $\vecell=(\ell_1,\ldots,\ell_k)\in\RR_{>0}^k$, and let
$\scrD\subset\{ \vecx\in\RR^d: 0<x_1,\ldots,x_{d-1}\leq x_d\}$ 
be a bounded subset with boundary of Lebesgue measure zero.
Then for any bounded continuous function $f:X_k\to\RR$,
\begin{align}\label{equiThm3res}
\lim_{T\to\infty}\frac1{T^d}\sum_{(\veca,n)\in\widehat\NN^d\cap T\scrD} 
f\bigl(L_{n,\veca,\vecell}\bigr)
=\frac{\vol(\scrD)}{\zeta(d)} \int_{L\in X_k} f(L) \, d\mu_0(L).
\end{align}
\end{thm}

In order to prove Theorem \ref{equiThm3} we first prove
Theorem \ref{equiThm2} below, which is a corollary 
of \cite[Thm.\ 7]{Marklof10}.
Set $G=\SL(d,\R)$, $G_0=\SL(k,\R)$ and $\Gamma=\SL(d,\Z)$, $\Gamma_0=\SL(k,\Z)$.
For any $M\in G_0$, $\Z^kM$ is a $k$-dimensional lattice of covolume one
in $\R^k$, and this gives an identification of $X_k$ with the homogeneous
space $\Gamma_0\backslash G_0$.
Then $\mu_0$ is identified with the unique $G_0$-right invariant
probability measure on $\Gamma_0\backslash G_0$;
we also use the same notation $\mu_0$ for the corresponding 
Haar measure on $G_0$.
Let $H$ be the following subgroup of $G$:
\begin{equation}
H=\bigg\{ M=\begin{pmatrix} A & \trans\vecnull \\ \vecc & 1 \end{pmatrix} : A\in G_0, \; \vecc\in\RR^k\bigg\}.
\end{equation}
We normalize the Haar measure $\mu_H$ of $H$ so that it becomes a probability measure on $\Gamma\backslash\Gamma H$; explicitly:
\begin{equation} \label{siegel2}
d\mu_H(M) = d\mu_0(A)\, d\vecc, \qquad M=\begin{pmatrix} A & \trans\vecnull \\ \vecc & 1 \end{pmatrix},
\end{equation}
where $d\vecc$ denotes the standard Lebesgue measure on $\R^k$.
We set
\begin{align}
D_n'(\vecell):=\matr{D_n(\vecell)}{\trans\bn}{\bn}{n}
=\diag\bigl(\Pi^{-1/k}\ell_1,\ldots,\Pi^{-1/k}\ell_k,n\bigr)\in G
\qquad (\Pi=n \ell_1\cdots\ell_k).
\end{align}
\begin{thm}\label{equiThm2} 
\rule[0ex]{0ex}{0ex}%
\begin{enumerate}
\item[(i)] For every $(\veca,n)\in\hatNN^d$ we have
$u(n^{-1}\veca)D_n'(\vecell)\in\Gamma H$.
\item[(ii)] For any $\vecell\in\RR_{>0}^k$, any bounded subset
$\scrD\subset\{ \vecx\in\RR^d: 0<x_1,\ldots,x_{d-1}\leq x_d\}$ 
with boundary of Lebesgue measure zero, 
and any bounded continuous function $f_0:\Gamma\backslash\Gamma H\to\RR$,
we have
\begin{equation}\label{eqThm2}
	\lim_{T\to\infty} \frac{1}{T^d} \sum_{(\veca,n)\in\widehat\NN^d\cap T\scrD} 
f_0\Bigl(u(n^{-1}\veca)D_n'(\vecell)\Bigr)  
	= \frac{\vol(\scrD)}{\zeta(d)} \int_{\Gamma\backslash\Gamma H} f_0(M) \, d\mu_H(M)  .
\end{equation}
\end{enumerate}
\end{thm}
\begin{proof}
To prove (i), note that for any $(\veca,n)\in\hatNN^d$
there exists $\gamma\in\Gamma$ such that $(\veca,n)\gamma=\vece_d$
and for this $\gamma$ we have
\begin{equation}
\trans\gamma u(n^{-1}\veca)
=\matr A{\trans\bn}{\vecc}{n^{-1}}
\end{equation}
for some $\vecc\in\R^k$ and $A\in\GL(k,\R)$ with $\det A=n$.
It follows that $\trans\gamma u(n^{-1}\veca)D_n'(\vecell)\in H$,
and hence $u(n^{-1}\veca)D_n'(\vecell)\in \Gamma H$,

Next to prove (ii), note that since $\Gamma\backslash\Gamma H$
is an embedded submanifold of $\Gamma\backslash G$,
it suffices to prove that \eqref{eqThm2} holds when $f_0$ is an
arbitrary bounded continuous  real-valued function 
\textit{on $\Gamma\backslash G$.}
But this
follows by applying Theorem 7 in \cite{Marklof10} with the test function 
$f(\vecx,M)=f_0(M D'_{x_d}(\vecell))$.
\end{proof}

\begin{proof}[Proof of Theorem \ref{equiThm3}]
We have 
\begin{align}
L_{n,\veca,\vecell}=\bigl((\Z^k\times n\Z)u(\veca)\bigr)_0 D_n(\vecell)
=\bigl(\Z^{k+1}u(n^{-1}\veca)\bigr)_0 D_n(\vecell)
=\bigl(\Z^{k+1}u(n^{-1}\veca)D_n'(\vecell)\bigr)_0.
\end{align}
Hence the left hand side of \eqref{equiThm3res} can be expressed as
\begin{align}\label{equiThm3PF1}
\lim_{T\to\infty}\frac1{T^d}\sum_{(\veca,n)\in\widehat\NN^d\cap T\scrD} 
f\Bigl(\bigl(\Z^{k+1}u(n^{-1}\veca)D_n'(\vecell)\bigr)_0\Bigr).
\end{align}
Let us now apply Theorem \ref{equiThm2} with the test function $f_0$
given by
$f_0(M):=f\bigl((\Z^{k+1}M)_0\bigr)$ for all $M\in\Gamma H$.
To see that this is well-defined, note that if
$M=\gamma\smatr A{\trans\bn}{\vecc}1$ with $\gamma\in\Gamma$ 
and $\smatr A{\trans\bn}{\vecc}1\in H$ then
\begin{align}
(\Z^{k+1}M)_0=\left(\Z^{k+1}\matr A{\trans\bn}{\vecc}1\right)_0=\Z^kA
\in X_k
\end{align}
so that $f_0(M)=f(\Z^kA)$.
The function $f_0$ is obviously bounded and $\Gamma$-left invariant;
furthermore the formula $f_0(M)=f(\Z^kA)$ just proved shows that 
$f_0$ is continuous on $H$, and hence on $\Gamma H$.
Now Theorem \ref{equiThm2} gives that the limit in \eqref{equiThm3PF1} 
equals
\begin{align}
\frac{\vol(\scrD)}{\zeta(d)}\int_{\Gamma\backslash\Gamma H}f_0(M)\,d\mu_H(M)
=\frac{\vol(\scrD)}{\zeta(d)}\int_{\Gamma_0\backslash G_0}
\int_{\R^k/\Z^k}f_0\left(\matr A{\trans\bn}{\vecc}1\right)\,d\vecc\,d\mu_0(A)
\\\notag
=\frac{\vol(\scrD)}{\zeta(d)}\int_{\Gamma_0\backslash G_0}f(\Z^kA)\,d\mu_0(A),
\end{align}
and we are done.
\end{proof}

Theorems \ref{Thm1}  and \ref{Thm2} now follow from Theorem \ref{equiThm3} 
combined with \eqref{TDCARDINALITY}, Proposition \ref{DIAMFINALRELPROP}
and Lemma~\ref{DIAMASCOVRADLEM}.
Indeed, let $\vecell\in\R_{>0}^k$ and 
$\scrD\subset\fF^+$ be given as in Theorem \ref{Thm1}.
Then Theorem \ref{equiThm3} together with \eqref{TDCARDINALITY} imply 
that if we view $L_{n,\veca,\vecell}$ as a ($X_k$-valued) random variable
defined by taking $(\veca,n)$ uniformly at random in 
$\widehat\NN^{k+1}\cap T\scrD$, then as $T\to\infty$,
$L_{n,\veca,\vecell}$ converges in
distribution to a random variable $L\in X_k$ taken according to $\mu_0$.
We next note that the functions $L\mapsto\rho(\fP,L)$ and
$L\mapsto\rho(\Delta,L)$ are continuous on $X_k$
(this is immediate from \cite[Prop.\ 4.4]{rDvF2004};
for the case of $\Delta$ it was also proved in 
\cite[Lem.\ 4, Thm.\ 9]{Marklof10}).
Hence by the continuous mapping theorem,
\begin{align*}
\rho(\fP,L_{n,\veca,\vecell})\xrightarrow[]{\textup{ d }}\rho(\fP,L)
\quad\text{and}\quad
\rho(\Delta,L_{n,\veca,\vecell})\xrightarrow[]{\textup{ d }}\rho(\Delta,L)
\qquad\text{as }\: T\to\infty.
\end{align*}
Thus by Lemma \ref{DIAMASCOVRADLEM},
and using the obvious fact that
$(n\ell_1\cdots\ell_k)^{-\frac1k}\xrightarrow[]{\textup{ d }}0$,
we have both
\begin{align*}
&\diam^+_\vece(\R^k/L_{n,\veca,\vecell})
-\frac{\vece\cdot\vecell}{(n\ell_1\cdots\ell_k)^{1/k}}
\xrightarrow[]{\textup{ d }}\rho(\Delta,L),
\\
&\diam_\vece(\R^k/L_{n,\veca,\vecell})
-\frac{\vece\cdot\vecell}{2(n\ell_1\cdots\ell_k)^{1/k}}
\xrightarrow[]{\textup{ d }}\rho(\fP,L),
\quad\text{and}\quad
\diam_\vece(\R^k/L_{n,\veca,\vecell})\xrightarrow[]{\textup{ d }}\rho(\fP,L)
\end{align*}
as $T\to\infty$.
Hence Theorem \ref{Thm1} follows from
Proposition \ref{DIAMFINALRELPROP},
and so does Theorem \ref{Thm2} if we also assume $\scrD\subset\fF$.
\hfill $\square$ $\square$

\section{On the distribution of $\rho(\fP,L)$ for random $L\in X_k$}
\label{RHOPLDISTRSEC}

In this section we give proofs of those results about the
distribution of $\rho(\fP,L)$ for random $L\in X_k$ which we
have used or mentioned in previous sections.

\subsection{Proof of the continuity of $\tilde P_k(R)$}
\label{TILDEPCONTSEC}

The proof that $\tilde P_k(R)$ is a continuous function of $R$ follows
the same basic strategy as the proof of the continuity of
$P_k(R)=\Psi_d(R)$ in \cite[Lem.~7]{Marklof10},
but the details are a bit more complicated.
We start by giving a necessary criterion for $\rho(\fP,L)=R$,
in Lemma \ref{RHOPLCRITLEM} below.
We write $\{\pm1\}^k$ for the set of all vectors in $\R^k$ of the form
$\pm\vece_1\pm\vece_2\pm\ldots\pm\vece_k$. For each 
$\veceps=(\epsilon_1,\ldots,\epsilon_k)\in\{\pm1\}^k$
we let $\fP_\veceps$ be the (closed) face of $\fP$ given by
\begin{align}
\fP_\veceps &=\{\vecx\in\fP\col\veceps\cdot\vecx=1\}
\\\notag
&=\Bigl\{\vecx=(x_1,\ldots,x_k)\in\R^k\col \sum_{j=1}^k\epsilon_jx_j=1
\text{ and }\epsilon_jx_j\geq0\text{ for all }j=1,\ldots,k\Bigr\}.
\end{align}
It is clear from the last relation that $\fP_\veceps$ is a 
$(k-1)$-dimensional simplex.
The faces $\fP_\veceps$ together cover the boundary of $\fP$:
\begin{align}\label{FPBDRY}
\partial\fP=\Bigl\{\vecx\in\R^k\col\sum_{j=1}^k|x_j|=1\Bigr\}=
\bigcup_{\veceps\in\{\pm1\}^k}\fP_\veceps.
\end{align}
\begin{lem}\label{RHOPLCRITLEM}
If $\rho(\fP,L)=R$ for some $L\in X_k$ and $R>0$, then there is a vector
$\veczeta\in\R^k$ and a nonempty subset $E\subset\{\pm1\}^k$ such that
\begin{enumerate}
\item[(i)] $L\cap(\veczeta+R\fP^\circ)=\emptyset$;
\item[(ii)] 
$L\cap (\veczeta+R\fP_\veceps)\neq\emptyset$ for each $\veceps\in E$;
\item[(iii)] there does not exist any $\vecalf\in\R^k$
satisfying $\veceps\cdot\vecalf<0$ for all $\veceps\in E$.
\end{enumerate}
\end{lem}
\begin{proof}
Let $L\in X_k$ and $R>0$ be given with $\rho(\fP,L)=R$.
Then, similarly to what we noted in the proof of Lemma~\ref{DIAMASCOVRADLEM},
$R$ is %
the supremum of all $R'>0$ such that there exists a 
translate of $R'\fP$ which is disjoint from $L$.
Hence by a simple compactness argument there is some 
translate of $R\fP^\circ$ which is disjoint from $L$,
i.e.\ we have $L\cap(\veczeta+R\fP^\circ)=\emptyset$ for some 
$\veczeta\in\R^k$.
Let us fix such a vector $\veczeta$,
and let $E$ be the set of all $\veceps\in\{\pm1\}^k$ for which
$L\cap (\veczeta+R\fP_\veceps)\neq\emptyset$. Then conditions (i) and (ii)
hold by construction. 
Assume that (iii) does \textit{not} hold,
and let $\vecalf$ be a vector in $\R^k$ satisfying
$\veceps\cdot\vecalf<0$ for all $\veceps\in E$.
Now because of $L\cap(\veczeta+R\fP^\circ)=\emptyset$
and \eqref{FPBDRY}, for every point $\vecx\in(\veczeta+R\fP)\cap L$
there exists some $\veceps\in E$ such that $\vecx\in\veczeta+R\fP_\veceps$.
In particular we then have $\veceps\cdot(\vecx-\veczeta)=R$,
and hence, for all $t>0$,
\begin{align}
\veceps\cdot(\vecx-(\veczeta+t\vecalf))=R-t\veceps\cdot\vecalf>R,
\end{align}
so that $\vecx\notin(\veczeta+t\vecalf)+R\fP$.
It follows that $(\veczeta+t\vecalf)+R\fP$ is disjoint from
$(\veczeta+R\fP)\cap L$, for every $t>0$.
Hence for $t>0$ sufficiently small,
$(\veczeta+t\vecalf)+R\fP$ is in fact disjoint from all $L$,
so that $L\cap((\veczeta+t\vecalf)+R'\fP)=\emptyset$
even holds for some $R'>R$.
This gives $\rho(\fP,L)>R$, a contradiction.
Hence also condition (iii) must hold.
\end{proof}

\begin{lem}\label{RHOPLCRITINTERPRLEM}
A finite nonempty subset 
$E=\{\veceps_1,\ldots,\veceps_r\}\subset\R^k\setminus\{\bn\}$ 
satisfies the condition {\rm (iii)} in Lemma \ref{RHOPLCRITLEM}
if and only if $\sum_{i=1}^r c_i\veceps_i=\bn$ holds for some
choice of $c_1,\ldots,c_r\geq0$, not all $0$.
\end{lem}
\begin{proof}
Let $C$ be the conic hull of $-E$.
Then condition (iii) in Lemma \ref{RHOPLCRITLEM} says
that the dual cone $C^*$ has empty interior,
or in other words that $C^*$ is contained in a proper linear subspace of $\R^k$.
This holds if and only if $C^{**}=C$ contains a line through the
origin, i.e.\ if and only if
$\sum_{i=1}^r c_i\veceps_i=\bn$ holds for some
choice of $c_1,\ldots,c_r\geq0$, not all $0$.
\end{proof}

The following lemma shows that $\tilde P_k(R)$ is continuous.
\begin{lem}\label{TILDEPCONTLEM}
For every $R>0$, 
\begin{align}
\mu_0(\{L\in X_k\col\rho(\fP,L)=R\})=0.
\end{align}
\end{lem}
\begin{proof}
By the definition of $\mu_0$, it is equivalent to prove that
the set %
$S$ of all $A\in G_0$ satisfying $\rho(\fP,\Z^kA)=R$
satisfies $\mu_0(S)=0$.
Let $\scrE$ be the family of all subsets $E\subset\{\pm1\}^k$ satisfying
condition (iii) in Lemma \ref{RHOPLCRITLEM}.
Then, by that lemma, $S$ is a subset of
\begin{align}
\bigcup_{E\in\scrE}\bigl\{A\in G_0\col \text{ there exists $\veczeta\in\R^k$
such that }\Z^kA\cap(\veczeta+R\fP_\veceps)\neq\emptyset,\:\forall\veceps\in E
\bigr\}.
\end{align}
But $\scrE$ is finite; hence it suffices to prove that each
individual set in the above union has measure zero.
Thus fix some $E\in\scrE$; say $E=\{\veceps_1,\ldots,\veceps_r\}$.
The corresponding set in the above union can be expressed as
\begin{align}
\bigcup_{\vecn_1,\ldots,\vecn_r\in\Z^k}
\bigl\{A\in G_0\col\text{ there exists $\veczeta\in\R^k$
such that }\vecn_iA\in\veczeta+R\fP_{\veceps_i}\:
(i=1,\ldots,r)\bigr\}.
\end{align}
This is a countable union, and hence it suffices to prove that each
individual set in the union has measure zero. Thus we fix some 
$\vecn_1,\ldots,\vecn_r\in\Z^k$.
Since $E$ satisfies condition (iii) in Lemma \ref{RHOPLCRITLEM},
there exist, by Lemma \ref{RHOPLCRITINTERPRLEM}, some
$c_1,\ldots,c_r\geq0$, not all zero, so that
$\sum_{i=1}^r c_i\veceps_i=\bn$.
Now $\vecn_iA\in\veczeta+R\fP_{\veceps_i}$ implies
$(\vecn_iA-\veczeta)\cdot\veceps_i=R$, and multiplying this relation with
$c_i$ and adding over all $i$ we obtain
$\sum_{i=1}^rc_i\vecn_iA\cdot\veceps_i=R\sum_{i=1}^rc_i$.
Hence the set 
corresponding to our fixed $\vecn_1,\ldots,\vecn_r$ 
in the above union
is a subset of:
\begin{align}\label{TILDEPCONTLEMPF1}
\Bigl\{A\in G_0\col\sum_{i=1}^rc_i\vecn_iA\cdot\veceps_i=R\sum_{i=1}^rc_i
\Bigr\}
=\Bigl\{A\in G_0\col\text{tr}(MA)=R\sum_{i=1}^rc_i\Bigr\},
\end{align}
where $M=(m_{\ell j})$ is the $k\times k$-matrix given by
$m_{\ell j}=\sum_{i=1}^rc_i(\vecn_i\cdot\vece_j)(\veceps_i\cdot\vece_\ell)$.
We have $\sum_{i=1}^rc_i>0$, since $c_1,\ldots,c_r\geq0$ and
at least one $c_i$ is positive.
Hence if $M=0$ then the set \eqref{TILDEPCONTLEMPF1} is empty.
If $M\neq0$ then the set \eqref{TILDEPCONTLEMPF1} is a submanifold of
$G_0$ of codimension one
(cf.\ the proof of \cite[Lem.\ 7]{Marklof10}).
Hence the set \eqref{TILDEPCONTLEMPF1} 
has measure zero also in this case and the proof is complete.
\end{proof}

\subsection{Proof of $\tilde\rho_k>\frac12(k!)^{1/k}$ for $k\geq3$}
\label{TILDERHOSTRICTINEQSEC}

We noted in \eqref{TILDERHOKLOWBOUND} that 
$\tilde\rho_k\geq\vol(\fP)^{-1/k}=\sfrac12(k!)^{1/k}$
and in the present section we will prove that \textit{strict} inequality holds
in this relation when $k\geq3$.
Since the infimum in \eqref{TILDERHOKDEF} is known to be attained
(cf., e.g., \cite[Thm.\ 21.3]{pGcL87}),
it suffices to prove that there does not exist any lattice
covering of $\R^k$ by translates of $\fP$ which has density exactly one,
viz.\ with the $\fP$-translates having pairwise disjoint interiors.
In fact we will prove the stronger fact that there does not
exist any tessellation (lattice or non-lattice) of $\R^k$ 
by translates of $\fP$:
\begin{prop}\label{NOTESSELLATIONPROP}
For $k\geq3$ there does not exist any subset $P\subset\R^k$
such that $P+\fP=\R^k$ and $(\vecr+\fP^\circ)\cap(\vecs+\fP^\circ)=\emptyset$
for all $\vecr\neq\vecs\in P$.
Hence in particular, $\tilde\rho_k>\frac12(k!)^{1/k}$ for $k\geq3$.
\end{prop}

The proof of this fact is quite easy but we have not been able to find an
appropriate reference for it.
The question of finding the optimal lattice covering of $\R^3$
by translates of $\fP$ was studied by Dougherty and Faber in
\cite[Sec.\ 7]{rDvF2004}, and they conjecture that the optimal density
is $\frac98$, which would mean that
$\tilde\rho_3=\frac34\sqrt[3]2=0.9449\ldots$.
We remark that for the more classical question of lattice \textit{sphere}
coverings, the optimal coverings are known in dimensions up to $5$;
cf.\ \cite{sDaSfV2008}, \cite{sReB76}, \cite{fV2003}.

\begin{proof}[Proof of Proposition \ref{NOTESSELLATIONPROP}]
Assume $P+\fP=\R^k$ and $(\vecr+\fP^\circ)\cap(\vecs+\fP^\circ)=\emptyset$
for all $\vecr\neq\vecs\in P$.
Without loss of generality we assume $\bn\in P$.
Now for any point $\vecx=(x_1,\ldots,x_k)\in\R_{>0}^d$ with
$x_1+\ldots+x_k=1$ (i.e. $\vecx\in\partial\fP$) we may argue as follows.
For any $\ve>0$ we have $\vecx+\ve\vece\notin\fP$, and thus
$\vecx+\ve\vece\in\vecr+\fP$ for some 
$\vecr=(r_1,\ldots,r_k)\in P\setminus\{\bn\}$.
Letting $\ve\to0$ it follows that there exists a point
$\vecr=(r_1,\ldots,r_k)\in P\setminus\{\bn\}$ such that
$\vecx\in\vecr+\partial\fP$,
i.e.\ $\sum_{j=1}^k|r_j-x_j|=1$, and also $\vecx+\ve\vece\in\vecr+\fP$ for all
$\ve$'s in some sequence of positive numbers tending to $0$.
Now
\begin{align}
\sum_{j=1}^k|r_j|\leq\sum_{j=1}^k|r_j-x_j|+\sum_{j=1}^k x_j=1+1=2,
\end{align}
and if $r_j<x_j$ would hold for some $j$ then we would have strict
inequality in the above computation, and this would lead to the contradiction
$\frac12\vecr\in\fP^\circ\cap(\vecr+\fP^\circ)$.
To sum up, we have proved that for any given 
$\vecx\in\R_{>0}^k$ with $x_1+\ldots+x_k=1$,
there exists some $\vecr\in P\setminus\{\bn\}$ satisfying
$r_j\geq x_j$ for $j=1,\ldots,k$, and $\sum_{j=1}^k r_j=2$.

Let us first apply the above fact with $\vecx=(1-(k-1)\ve,\ve,\ldots,\ve)$
with $\ve>0$ tending to zero.
It follows that there exists some $\vecr\in P$ with $r_1\geq1$,
$r_2,\ldots,r_k>0$ and $\sum_{j=1}^k r_j=2$.
Next we apply the above fact with
$\vecx=(r_1-1+\ve,r_2+\ve,r_3-2\ve,r_4,\ldots,r_k)$
(here we use $k\geq3$!).
This leads to the conclusion that there exists some 
$\vecs\in P$ with $s_1>r_1-1$, $s_2>r_2$, $s_j\geq r_j$ 
for $j=3,\ldots,k$, and $\sum_{j=1}^ks_j=2$.
In particular $\vecr\neq\vecs$ since $s_2>r_2$.
Now $s_1=2-\sum_{j=2}^ks_j<2-\sum_{j=2}^kr_j=r_1$, and hence
\begin{align}
\sum_{j=1}^k|s_j-r_j|=(r_1-s_1)+\sum_{j=2}^k(s_j-r_j)
=2r_1-2s_1<2r_1-2(r_1-1)=2,
\end{align}
which leads to the contradiction
$\frac12(\vecr+\vecs)\in(\vecr+\fP^\circ)\cap(\vecs+\fP^\circ)$.
\end{proof}

\subsection{The asymptotic formula for $\tilde P_k(R)$}
\label{PTILDEASYMPTSEC}

We now discuss the proof of the asymptotic formula stated in 
Remark \ref{PTILDEASYMPTREM}, viz.\ 
\begin{align}\label{PTILDEASYMPT}
\tilde P_k(R)=\frac{R^{-k}}{2\zeta(k)}+O_k\bigl(R^{-k-1-\frac1{k-1}}\bigr).
\end{align}
It turns out that most of the proof in
\cite{strombergsson11} of the asymptotic formula for $P_k(R)$,
\eqref{PKASYMPT}, carries over with very small changes to the present case:
Mimicking \cite[Sec.\ 2.1-3]{strombergsson11} we obtain
\begin{align}
\tilde P_k(R)=\frac{R^{-k}}{2k\zeta(k)}
\int_{\S_1^{k-1}}\ell(\vecv)^{-k}\,d\vecv
+O_k(R^{-(k+1)-\frac1{k-1}}),
\end{align}
where $\S_1^{k-1}$ is the unit sphere in $\R^k$ centered at zero, 
$d\vecv$ is the $(k-1)$-dimensional volume measure on $\S_1^{k-1}$,
and $\ell(\vecv)$ is the width of $\fP$ in the direction $\vecv$,
viz., for $\vecv=(v_1,\ldots,v_d)\in\S_1^{k-1}$,
\begin{align}
\ell(\vecv)=2\max(|v_1|,\ldots,|v_k|).
\end{align}
Now to get \eqref{PTILDEASYMPT} it only remains to prove the following.
\begin{lem}\label{PTILDECONSTLEM}
For every $k\geq2$ we have
\begin{align}
\int_{\S_1^{k-1}}\ell(\vecv)^{-k}\,d\vecv=k.
\end{align}
\end{lem}
\begin{proof}
Let $\fP^*$ be the polar body of $\fP$, i.e.\ 
\begin{align}
\fP^*=\bigl\{\vecx\in\R^k\col\vecx\cdot\vecy\leq1,\:\forall\vecy\in\fP\bigr\}
=\bigl\{r\vecv\col\vecv\in\S_1^{k-1},\:0\leq r\leq(\sfrac12\ell(\vecv))^{-1}
\bigr\}.
\end{align}
Then clearly
\begin{align}
\int_{\S_1^{k-1}}\ell(\vecv)^{-k}\,d\vecv=2^{-k}k\vol(\fP^*).
\end{align}
However one verifies easily that $\fP^*$ equals the $k$-dimensional cube
$[-1,1]^k$; hence $\vol(\fP^*)=2^k$ and the lemma follows.
\end{proof}

One may note that for $k=2$, \eqref{PTILDEASYMPT} says
$\tilde P_2(R)=\frac3{\pi^2}R^{-2}+O(R^{-4})$,
which is consistent with %
the explicit formula stated in Remark \ref{TILDEPEXPLREM}.

\section{The explicit formulas for $\tilde p_2(R)$ and $p_2(R)$}
\label{secExplicit}

We now prove the explicit formula for the density $\tilde p_2(R)$
which we stated in Remark \ref{TILDEPEXPLREM}.
\begin{prop}\label{TILDEPEXPLPROP}
For $k=2$ the density $\tilde p_k(R)=-\frac{d}{dR}\tilde P_k(R)$ is
given by
\begin{equation}\label{TILDEPEXPLPROPRES}
\tilde p_2(R)=
\begin{cases} 0 & (0\leq R\leq\frac{1}{\sqrt 2}) \\
\frac{24}{\pi^2} \big(\frac{2R^2-1}{R}\log\big(\frac{2R^2}{2R^2-1}\big)+\frac{1-R^2}{R}\log\big(\frac{R^2}{|1-R^2|}\big)\big) & (R>\frac{1}{\sqrt 2}) .
\end{cases}
\end{equation}
\end{prop}

\subsection{Auxiliary lemmas}\label{AUXLEMSEC}
To prepare for the proof of Proposition \ref{TILDEPEXPLPROP}
we first prove a 
series of lemmas.
As a first step, note that since 
$\fP$ 
for $k=2$
is a square with side $\sqrt2$, 
the formula \eqref{TILDEPKDEF} may be rewritten as
(using the $\SO(2)$-invariance of $\mu_0$)
\begin{align}\label{TILDEP2FORMULA}
\tilde P_2(R)=\mu_0\bigl(\bigl\{L\in X_2\col\rho\bigl(K,L\bigr)>r
\bigr\}\bigr),\qquad
\text{where }\: r:=\sqrt2 R,
\end{align}
and where $K$ is the unit square
\begin{align}
K:=[0,1]^2=\{\vecx=(x_1,x_2)\col x_1,x_2\in[0,1]\}.
\end{align}
We will make frequent use of the fact that,
just as in proof of Lemma~\ref{DIAMASCOVRADLEM}, $\rho(K,L)$ is the supremum of 
all $r>0$ for which there exists a translate of $rK$ that is disjoint from $L$.

We now %
introduce a parametrization of $X_2$ that is tailored
to give a practicable expression for \eqref{TILDEP2FORMULA}.
To motivate our definition below, 
note that by Lemma \ref{RHOPLCRITLEM}
(transformed from $\fP$ to $K$), if $L\in X_2$ and $\rho(K,L)=r$ then there
is some $\veczeta\in\R^2$ such that $L$ has 
no point in the interior of $\veczeta+rK$, but $L$ has a point on 
each of two opposite sides of $\veczeta+rK$.
By perturbing $\veczeta$ in a direction parallel to these sides
we may also, at least for generic $L$, assume that $L$ %
intersects one more side of $\veczeta+rK$.
If we assume that the three $L$-points on the sides of $\veczeta+rK$ are
$\veczeta+r(\alpha,0)$, $\veczeta+r(0,\beta)$ and $\veczeta+r(1,\gamma)$
with $\alpha,\beta,\gamma\in(0,1)$ then it follows that
$L$ contains the vectors $r(\alpha,-\beta)$ and $r(1,\gamma-\beta)$,
and in fact these two vectors necessarily span $L$, since
$L$ is disjoint from the interior of $\veczeta+rK$.
Using also the fact that $L$ has co-area one, it follows that
\begin{align}\label{LABCDEF}
L=L_{(\alpha,\beta,\gamma)}:=\delta^{-\frac12}\bigl(\Z(\alpha,-\beta)
+\Z(1,\gamma-\beta)\bigr)
\end{align}
where
\begin{align}
\delta=\delta(\alpha,\beta,\gamma):=
\left|\begin{matrix}\alpha&-\beta\\1&\gamma-\beta\end{matrix}\right|
=(1-\alpha)\beta+\alpha\gamma>0.
\end{align}
\begin{lem}\label{LOCALDIFFEOLEM}
The map $(\alpha,\beta,\gamma)\mapsto L_{(\alpha,\beta,\gamma)}$ 
is a local diffeomorphism from $(0,1)^3$ to $X_2$,
under which the measure $\mu_0$ corresponds to
\begin{align}\label{LOCALDIFFEOLEMRES}
\frac3{\pi^2}\delta(\alpha,\beta,\gamma)^{-2}\,d\alpha\,d\beta\,d\gamma.
\end{align}
\end{lem}
\begin{proof}
Set
\begin{align}
A=A_{(\alpha,\beta,\gamma)}:=\delta^{-\frac12}
\begin{pmatrix}\alpha&-\beta\\1&\gamma-\beta\end{pmatrix}\in G_0,
\end{align}
so that $L_{(\alpha,\beta,\gamma)}=\Z^2A$.
A computation shows that
the Iwasawa decomposition of $A$ %
is given by
\begin{align}
A_{(\alpha,\beta,\gamma)}=\matr 1x01\matr{\sqrt y}00{1/\sqrt y}
\matr{\cos\phi}{-\sin\phi}{\sin\phi}{\cos\phi}
\end{align}
where
\begin{align}\label{LOCALDIFFEOLEMPF1}
x=\frac{\alpha-\beta\gamma+\beta^2}{1+(\beta-\gamma)^2};
\qquad
y=\frac{(1-\alpha)\beta+\alpha\gamma}{1+(\beta-\gamma)^2};
\qquad
\phi=\frac\pi2+\arctan(\beta-\gamma).
\end{align}
One furthermore computes %
\begin{align}\label{LOCALDIFFEOLEMPF2}
\frac{\partial(x,y,\phi)}{\partial(\alpha,\beta,\gamma)}
=-(1+(\beta-\gamma)^2)^{-2}.
\end{align}
It is clear from \eqref{LOCALDIFFEOLEMPF1} that $x,y,\phi$ are
smooth functions of $(\alpha,\beta,\gamma)\in(0,1)^3$,
and since also the Jacobian determinant \eqref{LOCALDIFFEOLEMPF2}
is non-vanishing for all these $(\alpha,\beta,\gamma)$ it follows 
that the map $(\alpha,\beta,\gamma)\mapsto(x,y,\phi)$
is a local diffeomorphism from $(0,1)^3$ to $\R\times\R_{>0}\times(0,\pi)$.
However the Iwasawa decomposition is known to be a diffeomorphism from
$(x,y,\phi)\in\R\times\R_{>0}\times(\R/2\pi\Z)$ onto $G_0$, under which
the measure $\mu_0$ corresponds to
$\frac3{\pi^2}y^{-2}\,dx\,dy\,d\phi$.
Hence the map 
$(\alpha,\beta,\gamma)\mapsto A_{(\alpha,\beta,\gamma)}$ is a 
local diffeomorphism from $(0,1)^3$ to $G_0$, under which $\mu_0$ corresponds
to \eqref{LOCALDIFFEOLEMRES}.
To complete the proof of the lemma we need only recall that
the quotient map $G_0\to X_2=\Gamma_0\backslash G_0$
is a local diffeomorphism
and %
$\mu_0$ on $X_2$ is just the measure 
corresponding to $\mu_0$ on $G_0$.
\end{proof}

Set
\begin{align}\label{LPDEF}
L'_{(\alpha,\beta,\gamma)}:=\Z(\alpha,-\beta)+\Z(1,\gamma-\beta)\subset\R^2
\end{align}
so that 
$L_{(\alpha,\beta,\gamma)}=\delta^{-\frac12}L'_{(\alpha,\beta,\gamma)}$.
By construction the translated lattice
$(0,\beta)+L'_{(\alpha,\beta,\gamma)}$ contains three points on the
boundary of the unit square $K$, namely
$(\alpha,0)$, $(0,\beta)$ and $(1,\gamma)$.
We next determine those $(\alpha,\beta,\gamma)$ for which
$(0,\beta)+L'_{(\alpha,\beta,\gamma)}$ contains no \textit{other} point
in $K$.
\begin{lem}\label{ONLYTHREECONDLEM}
Given $(\alpha,\beta,\gamma)\in(0,1)^3$, the relation
\begin{align}\label{ONLYTHREECONDLEMCOND}
\bigl((0,\beta)+L'_{(\alpha,\beta,\gamma)}\bigr)\cap K
=\bigl\{(\alpha,0),(0,\beta),(1,\gamma)\bigr\}
\end{align}
holds if and only if $\beta+\gamma>1$.
\end{lem}
\begin{proof}
We have
\begin{align}\label{ONLYTHREECONDLEMPF1}
(0,\beta)+L'_{(\alpha,\beta,\gamma)}=
\bigl\{(0,\beta)+n_1(\alpha,-\beta)+n_2(1,\gamma-\beta)\col
\vecn=(n_1,n_2)\in\Z^2\bigr\}.
\end{align}
In this representation the three points $(\alpha,0)$, $(0,\beta)$, $(1,\gamma)$
correspond to $\vecn=(1,0)$, $\vecn=(0,0)$ and $\vecn=(0,1)$,
respectively.
Taking $\vecn=(-1,1)$ in \eqref{ONLYTHREECONDLEMPF1} we see that 
\eqref{ONLYTHREECONDLEMCOND} implies $(1-\alpha,\beta+\gamma)\notin K$,
viz.\ $\beta+\gamma>1$.
Conversely, assume $\beta+\gamma>1$.
One then immediately checks that, in the above representation,
those $\vecn$ with $|n_1|\leq1$ which give points in $K$ are
$\vecn=(1,0),(0,0),(0,1)$, and no others.
To conclude the proof of the lemma it now suffices to show that
all $\vecn\in\Z^2$ with $|n_1|\geq2$ also give points outside $K$.
Assume the opposite, i.e.\ that 
\begin{align}
\vecp:=(0,\beta)+n_1(\alpha,-\beta)+n_2(1,\gamma-\beta)\in K
\end{align}
for some $\vecn\in\Z^2$ with $|n_1|\geq2$.
It is a simple geometric fact that for any such $\vecn$, there exists an integer
$m$ such that the point $(\sgn(n_1),m)$ belongs to the closed triangle
with vertices $(0,0)$, $(0,1)$ and $\vecn$.
Applying the affine map 
$(x,y)\mapsto(0,\beta)+x(\alpha,-\beta)+y(1,\gamma-\beta)$ we conclude that
the point
\begin{align}
\vecq:=(0,\beta)+\sgn(n_1)(\alpha,-\beta)+m(1,\gamma-\beta)
\end{align}
lies in the closed triangle with vertices $(0,\beta)$, $(1,\gamma)$, $\vecp$.
Hence, since $\vecq$ is not equal to one of the triangle vertices,
and since $0<\beta,\gamma<1$ and $\vecp\in K$,
we conclude that $\vecq\in K^\circ$.
This is a contradiction since we saw above that no point in 
\eqref{ONLYTHREECONDLEMPF1} with $|n_1|\leq1$ lies in $K^\circ$.
\end{proof}
Set 
\begin{align}
\Omega:=\bigl\{(\alpha,\beta,\gamma)\in(0,1)^3\col\beta+\gamma>1\bigr\}.
\end{align}
After a translation and a scaling, 
Lemma \ref{ONLYTHREECONDLEM} says that for any 
$(\alpha,\beta,\gamma)\in\Omega$, the lattice
$L_{(\alpha,\beta,\gamma)}$
meets $\delta^{-\frac12}(0,-\beta)+\delta^{-\frac12}K$ in exactly
three points, all lying on the boundary of this square.
Hence for such $(\alpha,\beta,\gamma)$ we have
$\rho(K,L_{(\alpha,\beta,\gamma)})\geq\delta^{-\frac12}$.
The next lemma shows that we always have equality in this relation.
\begin{lem}\label{RHOKLEQULEM}
For any $(\alpha,\beta,\gamma)\in\Omega$
we have $L'_{(\alpha,\beta,\gamma)}\cap(\veczeta+rK^\circ)\neq\emptyset$
for all $r>1$, $\veczeta\in\R^2$, and hence
$\rho(K,L_{(\alpha,\beta,\gamma)})=\delta^{-\frac12}$.
\end{lem}
\begin{proof}
Assume the contrary; then
$L'_{(\alpha,\beta,\gamma)}\cap(\veczeta+rK^\circ)=\emptyset$
for some $r>1$ and $\veczeta\in\R^2$.
Note that there exists $t>0$ such that
$L'_{(\alpha,\beta,\gamma)}\cap(\veczeta-(0,t)+rK^\circ)\neq\emptyset$
(this follows since $L'_{(\alpha,\beta,\gamma)}$ contains a vector
with positive $\vece_1$-component $<r$, e.g.\ the vector $(\alpha,0)$).
Taking $t_0\geq0$ to be the infimum of all $t>0$ with that property,
and then replacing $\veczeta$
with $\veczeta-(0,t_0)$, we obtain a situation where
the side $\{\veczeta+(x,0)\col 0<x<r\}$ contains a lattice point
$\vecell\in L'_{(\alpha,\beta,\gamma)}$,
while still $L'_{(\alpha,\beta,\gamma)}\cap(\veczeta+rK^\circ)=\emptyset$.
But now also $\vecell+(-\alpha,\beta)\in L'_{(\alpha,\beta,\gamma)}$
and $\vecell+(1-\alpha,\gamma)\in L'_{(\alpha,\beta,\gamma)}$,
and at least one of these two points must lie in $\veczeta+rK^\circ$,
since $\vecell\in\{\veczeta+(x,0)\col 0<x<r\}$ and $r>1$.
This is a contradiction.
\end{proof}

\begin{lem}\label{DIFFEOLEM}
The map $(\alpha,\beta,\gamma)\mapsto L_{(\alpha,\beta,\gamma)}$
is a diffeomorphism from $\Omega$ 
onto an open subset $X_2'$ of $X_2$.
\end{lem}
\begin{proof}
In view of Lemma \ref{LOCALDIFFEOLEM} it suffices to prove that
the map is injective. Thus assume
$L_{(\alpha,\beta,\gamma)}=L_{(\alpha',\beta',\gamma')}$
for some $(\alpha,\beta,\gamma),(\alpha',\beta',\gamma')\in\Omega$.
Then $\delta(\alpha,\beta,\gamma)=\delta(\alpha',\beta',\gamma')$
by Lemma \ref{RHOKLEQULEM}, and hence
$L'_{(\alpha,\beta,\gamma)}=L'_{(\alpha',\beta',\gamma')}$.
Call this lattice $L'$.
Using now Lemma \ref{ONLYTHREECONDLEM} and
$(\alpha,-\beta)\in L'$ it follows that $(\alpha,0)+L'$ is disjoint from
$K^\circ$.
In particular $(\alpha,0)+(-\alpha',\beta')\notin K^\circ$ and
$(\alpha,0)+(1-\alpha',\gamma')\notin K^\circ$,
and these two relations together imply $\alpha'=\alpha$.
Now also $\beta'=\beta$ and $\gamma'=\gamma$ follow easily.
\end{proof}

Let $W=\smatr0110$; this element acts on $\R^2$ by switching coordinates,
and it acts on lattices $L\subset\R^2$ by
$L\mapsto LW:=\{\vecx W\col \vecx\in L\}$.
The latter action gives a diffeomorphism of $X_2$ onto itself,
preserving $\mu_0$.
Set
\begin{align}
X_2'':=X_2'W,
\end{align}
where $X_2'$ is the open subset of $X_2$ defined in Lemma \ref{DIFFEOLEM}.
\begin{lem}\label{X0PX0PPDISJLEM}
$X_2'\cap X_2''=\emptyset.$
\end{lem}
\begin{proof}
Assume the contrary; then 
$L_{(\alpha,\beta,\gamma)}=L_{(\alpha',\beta',\gamma')}W$ for some
$(\alpha,\beta,\gamma),(\alpha',\beta',\gamma')\in\Omega$.
Now $\rho(K,L_{(\alpha',\beta',\gamma')}W)=
\rho(K,L_{(\alpha',\beta',\gamma')})$, since $W$ maps $K$ onto itself;
hence by Lemma \ref{RHOKLEQULEM} we have
$\delta(\alpha,\beta,\gamma)=\delta(\alpha',\beta',\gamma')$, and thus also
$L'_{(\alpha,\beta,\gamma)}=L'_{(\alpha',\beta',\gamma')}W$.
Call this lattice $L'$.
By Lemma \ref{ONLYTHREECONDLEM} we have
\begin{align}\label{X0PX0PPDISJLEMPF1}
((0,\beta)+L')\cap K=\bigl\{(\alpha,0),(0,\beta),(1,\gamma)\bigr\}.
\end{align}
Using here $(\gamma',1-\alpha')\in L'$ we get
$(0,\beta)+(\gamma',1-\alpha')\notin K$, viz.\ $\beta>\alpha'$.
On the other hand using $(\beta',-\alpha')\in L'$ we get
that $(0,\beta)+(\beta',-\alpha')$ is either outside $K$ or
else equals $(\alpha,0)$; hence we must have $\beta\leq\alpha'$.
This is a contradiction.
\end{proof}

\begin{lem}\label{TOTALMEASUREONELEM}
$\mu_0(X_2'\cup X_2'')=1.$
\end{lem}
\begin{proof}
We have
\begin{align}
\mu_0(X_2'\cup X_2'')=2\mu_0(X_2')=
\frac6{\pi^2}\int_\Omega\delta^{-2}  %
\,d\alpha\,d\beta\,d\gamma
=\frac6{\pi^2}\int_\Omega
\frac{d\alpha\,d\beta\,d\gamma}{((1-\alpha)\beta+\alpha\gamma)^2},
\end{align}
by Lemma \ref{LOCALDIFFEOLEM}.
Writing this as an iterated integral and evaluating the innermost
integral over $\alpha\in(0,1)$, we get
\begin{align}
=\frac6{\pi^2}\int_0^1\int_{1-\gamma}^1
\frac{d\beta\,d\gamma}{\beta\gamma}
=-\frac6{\pi^2}\int_0^1\frac{\log(1-\gamma)}\gamma\,d\gamma
=\frac6{\pi^2}\int_0^\infty\frac x{e^x-1}\,dx
=\frac6{\pi^2}\Gamma(2)\zeta(2)=1.
\end{align}
(We substituted $\gamma=1-e^{-x}$ and then used \cite[Thm.\ 14]{aI32}.)
\end{proof}

\begin{remark}
Another way to prove Lemma \ref{TOTALMEASUREONELEM} is to make
the discussion preceding \eqref{LABCDEF} more precise,
so as to show
that for a generic lattice $L\in X_2$, 
there exists some $\veczeta\in\R^2$ such that
$L\cap(\veczeta+rK^\circ)=\emptyset$ and 
\textit{either} 
$L$ contains the three points 
$\veczeta+r(\alpha,0)$, $\veczeta+r(0,\beta)$ and $\veczeta+r(1,\gamma)$
for some $(\alpha,\beta,\gamma)\in\Omega$
(thus $L\in X_2'$),
\textit{or}
$L$ contains the three points 
$\veczeta+r(0,\alpha)$, $\veczeta+r(\beta,0)$ and $\veczeta+r(\gamma,1)$
for some $(\alpha,\beta,\gamma)\in\Omega$
(in which case $L\in X_2''$).
However the above proof by direct computation also serves as a nice
consistency check of our set-up.
\end{remark}

\subsection{Proof of Proposition \ref*{TILDEPEXPLPROP}}
Using \eqref{TILDEP2FORMULA}, Lemmas \ref{LOCALDIFFEOLEM},
\ref{RHOKLEQULEM}, \ref{DIFFEOLEM}, \ref{X0PX0PPDISJLEM}, 
\ref{TOTALMEASUREONELEM}, and the fact that 
$L\mapsto LW$ preserves both $\mu_0$ and $\rho(K,L)$, we get:
\begin{align}\label{TILDEPEXPLPROPPF1}
\tilde P_2(R)=\frac6{\pi^2}\int_{\Omega_r}\delta(\alpha,\beta,\gamma)^{-2}
\,d\alpha\,d\beta\,d\gamma
=\frac6{\pi^2} 
\int_0^1\int_{1-\gamma}^1
\int_{I_{\beta,\gamma,r}}\frac{d\alpha}{\delta(\alpha,\beta,\gamma)^2}
\,d\beta\,d\gamma,
\end{align}
where $r=\sqrt2R$,
\begin{align}
\Omega_r:=\bigl\{(\alpha,\beta,\gamma)\in\Omega\col
\delta(\alpha,\beta,\gamma)<r^{-2}\bigr\},
\end{align}
and %
\begin{align}
I_{\beta,\gamma,r}=\bigl\{\alpha\in(0,1)\col\delta(\alpha,\beta,\gamma)<r^{-2}
\bigr\}.
\end{align}
Recalling $\delta(\alpha,\beta,\gamma)=(1-\alpha)\beta+\alpha\gamma$
we find that 
$I_{\beta,\gamma,r}=(0,\frac{r^{-2}-\beta}{\gamma-\beta})$ if
$\beta<r^{-2}<\gamma$,
$I_{\beta,\gamma,r}=(\frac{r^{-2}-\beta}{\gamma-\beta},1)$ if
$\gamma<r^{-2}<\beta$,
while 
$I_{\beta,\gamma,r}=(0,1)$ if $\beta,\gamma<r^{-2}$
and $I_{\beta,\gamma,r}=\emptyset$ if $\beta,\gamma>r^{-2}$.
Now it is easy to compute the derivative of the innermost integral in 
\eqref{TILDEPEXPLPROPPF1} with respect to $r$,
using the fact that $\delta(\alpha,\beta,\gamma)=r^{-2}$ 
for $\alpha=\frac{r^{-2}-\beta}{\gamma-\beta}$.
If $\beta<r^{-2}<\gamma$ then we get
\begin{align}
\frac d{dr}\int_{I_{\beta,\gamma,r}}
\frac{d\alpha}{\delta(\alpha,\beta,\gamma)^2}
=\Bigl(\frac d{dr}\frac{r^{-2}-\beta}{\gamma-\beta}\Bigr)\cdot(r^{-2})^{-2}
=-\frac{2r}{\gamma-\beta}.
\end{align}
Similarly when $\gamma<r^{-2}<\beta$ we get
\begin{align}
\frac d{dr}\int_{I_{\beta,\gamma,r}}
\frac{d\alpha}{\delta(\alpha,\beta,\gamma)^2}=-\frac{2r}{\beta-\gamma},
\end{align}
while if $\beta,\gamma<r^{-2}$ or $\beta,\gamma>r^{-2}$ then the derivative
vanishes.
Hence we obtain, using also the symmetry between $\beta$ and $\gamma$:
\begin{align}
\tilde p_2(R)=-\frac d{dR}\tilde P_2(R)
=-\sqrt2\frac d{dr}\tilde P_2(R)
=\frac{12\sqrt2}{\pi^2}
\iint\limits_{J_r}
\frac{2r}{\gamma-\beta}\,d\beta\,d\gamma,
\end{align}
where $J_r$ is the set of all pairs $(\beta,\gamma)\in(0,1)^2$
satisfying both $\beta+\gamma>1$ and $\beta<r^{-2}<\gamma$.
If $r\leq1$ then $J_r=\emptyset$, so that $\tilde p_2(R)=0$.
On the other hand if $r>\sqrt2$ then we get
\begin{align}
\tilde p_2(R)&=\frac{24\sqrt2}{\pi^2}r\int_0^{r^{-2}}\int_{1-\beta}^1
\frac{d\gamma\,d\beta}{\gamma-\beta}
=\frac{24\sqrt2}{\pi^2}r\int_0^{r^{-2}}\log\Bigl(\frac{1-\beta}{1-2\beta}\Bigr)
\,d\beta
\\\notag
&=\frac{12\sqrt2}{\pi^2}r\Bigl(
(1-2r^{-2})\log(1-2r^{-2})-2(1-r^{-2})\log(1-r^{-2})\Bigr).
\end{align}
Finally if $1<r<\sqrt2$ then we get
\begin{align}
\tilde p_2(R)&=\frac{24\sqrt2}{\pi^2}r\biggl(\int_0^{1-r^{-2}}\int_{1-\beta}^1
\frac{d\gamma\,d\beta}{\gamma-\beta}
+\int_{1-r^{-2}}^{r^{-2}}\int_{r^{-2}}^1\frac{d\gamma\,d\beta}{\gamma-\beta}
\biggr)
\\\notag
&=\frac{24\sqrt2}{\pi^2}r\biggl(\int_0^{1-r^{-2}}
\log\Bigl(\frac{1-\beta}{1-2\beta}\Bigr)\,d\beta
+\int_{1-r^{-2}}^{r^{-2}}\log\Bigl(\frac{1-\beta}{r^{-2}-\beta}\Bigr)\,d\beta
\biggr)
\\\notag
&=\frac{12\sqrt2}{\pi^2}r\Bigl(
(1-2r^{-2})\log(2r^{-2}-1)-2(1-r^{-2})\log(1-r^{-2})\Bigr).
\end{align}
Hence, recalling $r=\sqrt2R$, we obtain the formula stated in
\eqref{TILDEPEXPLPROPRES}.
\hfill$\square$

\subsection{The explicit formula for $p_2(R)$}
\label{p2REXPLSEC}

We next turn to the explicit formula for $p_2(R)$ which we
stated in \eqref{Ustinovs}.
This formula is due to Ustinov \cite{Ustinov10}, who proved it
by an argument involving Kloosterman sums and continued fractions.
We think it may be of interest 
to see an alternative derivation of \eqref{Ustinovs} based on
the definition of $P_2(R)$
in terms of Haar measure on the space of lattices, cf.\ \eqref{PKRDEF},
and so we give an outline of this argument here.

The overall structure of the argument is similar to
the previous case of $\tilde p_2(R)$.

For any $(\alpha,\beta,\gamma)\in(-\frac12,\frac12)^3$ we set
$\Lambda_{(\alpha,\beta,\gamma)}:=\kappa^{-\frac12}
\Lambda'_{(\alpha,\beta,\gamma)}$,
where
\begin{align}
\Lambda'_{(\alpha,\beta,\gamma)}:=
\Z(-\sfrac12+\gamma,-\alpha-\gamma)+\Z(\beta+\gamma,-\sfrac12-\gamma)
\end{align}
and
\begin{align}
\kappa=\kappa(\alpha,\beta,\gamma):=
\left|\begin{matrix}
-\sfrac12+\gamma&-\alpha-\gamma\\\beta+\gamma&-\sfrac12-\gamma
\end{matrix}\right|
=\sfrac14+\alpha\beta+\alpha\gamma+\beta\gamma>0.
\end{align}
The motivation of the above definition is that the translated lattice 
$(\frac12-\gamma,\frac12+\gamma)+\Lambda'_{(\alpha,\beta,\gamma)}$
constains the three points $(0,\frac12-\alpha)$, $(\frac12+\beta,0)$ 
and $(\frac12-\gamma,\frac12+\gamma)$ on the boundary of $\Delta$.

By a similar computation as in Lemma \ref{LOCALDIFFEOLEM} one proves
that the map $(\alpha,\beta,\gamma)\mapsto \Lambda_{(\alpha,\beta,\gamma)}$ 
is a local diffeomorphism from $(-\frac12,\frac12)^3$ to $X_2$,
under which the measure $\mu_0$ corresponds to
\begin{align}\label{LOCALDIFFEORES2}
\frac3{\pi^2}\kappa(\alpha,\beta,\gamma)^{-2}\,d\alpha\,d\beta\,d\gamma.
\end{align}

Next one proves analogues of Lemma \ref{ONLYTHREECONDLEM} and
Lemma \ref{RHOKLEQULEM}.
It is useful to assume that \textit{at least two of $\alpha,\beta,\gamma$
are positive}.
Note that for generic $(\alpha,\beta,\gamma)\in(-\frac12,\frac12)^3$
we can always get to this situation after possibly applying the map
$W=\smatr0110$ (cf.\ Section \ref{AUXLEMSEC});
this is because
$\Delta W=\Delta$, 
$\Lambda'_{(\alpha,\beta,\gamma)}W=\Lambda'_{(-\beta,-\alpha,-\gamma)}$
and
$\Lambda_{(\alpha,\beta,\gamma)}W=\Lambda_{(-\beta,-\alpha,-\gamma)}$.
It now turns out that if $(\alpha,\beta,\gamma)\in(-\frac12,\frac12)^3$
and at least two of $\alpha,\beta,\gamma$ are positive,
then the necessary and sufficient 
condition for 
$(\frac12-\gamma,\frac12+\gamma)+\Lambda'_{(\alpha,\beta,\gamma)}$
to contain no \textit{other} points in $\Delta$ than 
$(0,\frac12-\alpha)$, $(\frac12+\beta,0)$ 
and $(\frac12-\gamma,\frac12+\gamma)$, is:
\begin{align}\label{P2EXPLNEWPF1}
\alpha+\beta>0,\quad\alpha+\gamma>0,\quad\beta+\gamma>0.
\end{align}
(Note that, in the other direction, \eqref{P2EXPLNEWPF1} %
implies that at least two of $\alpha,\beta,\gamma$ are positive.)
Next, for any $(\alpha,\beta,\gamma)\in(-\frac12,\frac12)^3$ satisfying
\eqref{P2EXPLNEWPF1}, the necessary and sufficient condition 
for 
$\Lambda'_{(\alpha,\beta,\gamma)}\cap(\veczeta+r\Delta^\circ)\neq\emptyset$
to hold for all $r>1$, $\veczeta\in\R^2$,
is $\alpha+\beta+\gamma\leq\sfrac12$. Set
\begin{align}
\Omega:=\Bigl\{(\alpha,\beta,\gamma)\in(-\sfrac12,\sfrac12)^3\col
\alpha+\beta>0,\:
\alpha+\gamma>0,\:
\beta+\gamma>0,\:
\alpha+\beta+\gamma<\sfrac12\Bigr\}.
\end{align}
It then follows from the last statements that
$\rho(\Delta,\Lambda_{(\alpha,\beta,\gamma)})=\kappa^{-\frac12}$ holds
for all $(\alpha,\beta,\gamma)\in\Omega$.

It now follows by similar arguments as in Lemma \ref{DIFFEOLEM}
and Lemma \ref{X0PX0PPDISJLEM}
that the map $(\alpha,\beta,\gamma)\mapsto\Lambda_{(\alpha,\beta,\gamma)}$
is injective when restricted $\Omega$, and hence gives a 
diffeomorphism from $\Omega$ onto an open subset $X_2'\subset X_2$,
and furthermore that $X_2'$ is disjoint from $X_2'':=X_2'W$. 
Finally, it turns out that the union of 
$X_2'$ and $X_2''$ has full measure in $X_2$:
\begin{align}\label{P2EXPLNEWPF3}
\mu_0(X_2'\cup X_2'')=1.
\end{align}
(This can be proved either by a direct computation, cf.\ below,
or else by proving
that a generic lattice in $X_2$ indeed must belong to either $X_2'$ or $X_2''$.)

Using \eqref{PKRDEF} and the above facts, it follows that
\begin{align}\label{P2EXPLNEWPF2}
P_2(R)=\frac6{\pi^2}\int_{\Omega_R}\kappa(\alpha,\beta,\gamma)^{-2}
\,d\alpha\,d\beta\,d\gamma,
\end{align}
where now 
\begin{align}
\Omega_R:=\bigl\{(\alpha,\beta,\gamma)\in\Omega\col
\kappa(\alpha,\beta,\gamma)<R^{-2}\bigr\}.
\end{align}
We next introduce $s=\alpha+\beta+\gamma$ and
$t=\alpha^2+\beta^2+\gamma^2$ as new variables of integration in 
\eqref{P2EXPLNEWPF2}. Note that $0<s<\frac12$ for all 
$(\alpha,\beta,\gamma)\in\Omega$;
also $t\geq\frac13s^2$ by Cauchy's inequality.
Conversely, for given $s\in(0,\frac12)$ and $t\geq\frac13s^2$,
the set of corresponding points $(\alpha,\beta,\gamma)\in\R^3$
is the circle with center $\frac13(s,s,s)$ and 
radius $\sqrt{t-\frac13s^2}$ in the plane $\{\alpha+\beta+\gamma=s\}$,
and we may parametrize these points as
\begin{align}
(\alpha,\beta,\gamma)=\sfrac13(s,s,s)+\sqrt{t-\sfrac13s^2}\Bigl(
(\cos\omega)\vecb_1+(\sin\omega)\vecb_2\Bigr),\qquad\omega\in\R/2\pi\Z,
\end{align}
where $\vecb_1,\vecb_2$ is an arbitrary fixed orthonormal basis 
in the orthogonal complement of $(1,1,1)$ in $\R^3$.
Now $(\alpha,\beta,\gamma)\in\Omega$ holds if and only if
$0<s<\frac12$ and $\max(\alpha,\beta,\gamma)<s$,
and the latter condition is equivalent to $(\alpha,\beta,\gamma)$
lying inside a certain equilateral triangle with side $2\sqrt2s$
and center $\sfrac13(s,s,s)$ in the plane $\{\alpha+\beta+\gamma=s\}$.
This triangle has inradius $\sqrt{\frac23}s$ and circumradius 
$2\sqrt{\frac23}s$;
hence if $\sqrt{t-\frac13s^2}<\sqrt{\frac23}s$ 
(viz., $t<s^2$) then \textit{all} $\omega$ correspond to
points in $\Omega$, while if
$\sqrt{\frac23}s\leq\sqrt{t-\frac13s^2}<2\sqrt{\frac23}s$ 
(viz., $s^2\leq t<3s^2$) then certain subintervals of 
$\omega\in\R/2\pi\Z$ have to be removed, and the Lebesgue measure of 
those $\omega\in\R/2\pi\Z$ which correspond to points in $\Omega$ is
\begin{align}
2\pi-6\arctan\Bigl(\sqrt{\sfrac32}s^{-1}\sqrt{t-s^2}\Bigr).
\end{align}
Hence, using also
$\kappa(\alpha,\beta,\gamma)=\frac14+\frac12s^2-\frac12t$
and $|\frac{\partial(\alpha,\beta,\gamma)}{\partial(s,t,\omega)}|
=\frac1{2\sqrt3}$, we obtain:
\begin{align}\notag
P_2(R)=\frac{\sqrt3}{\pi^2}\biggl(2\pi\int_0^{\frac12}\int_{I_{s,R}}
\frac{1}{(\sfrac14+\sfrac12s^2-\sfrac12t)^2}\,dt\,ds
\hspace{170pt}
\\\label{P2EXPLNEWPF4}
+\int_0^{\frac12}\int_{J_{s,R}}
\frac{2\pi-6\arctan\Bigl(\sqrt{\sfrac32}s^{-1}\sqrt{t-s^2}\Bigr)}
{(\sfrac14+\sfrac12s^2-\sfrac12t)^2}\,dt\,ds\biggr),
\end{align}
where
\begin{align}
I_{s,R}=\bigl(\sfrac13s^2,s^2\bigr)\cap\bigl(s^2+\sfrac12-2R^{-2},\infty\bigr)
\quad\text{and}\quad
J_{s,R}=\bigl(s^2,3s^2\bigr)\cap\bigl(s^2+\sfrac12-2R^{-2},\infty\bigr).
\end{align}
In particular for $R\leq\sqrt3$ we have
$I_{s,R}=(\frac13s^2,s^2)$ and $J_{s,R}=(s^2,3s^2)$ for all
$s\in(0,\frac12)$ and in this case $P_2(R)=1$,
corresponding to the fact that 
the union of $X_2'$ and $X_2''$ has full measure in $X_2$,
cf.\ \eqref{P2EXPLNEWPF3}.
Next if $\sqrt3\leq R\leq2$ then still $J_{s,R}=(s^2,3s^2)$ for all 
$s\in(0,\frac12)$, but now
$I_{s,R}=(\frac13s^2,s^2)$ only for 
$s\in(0,\frac{\sqrt3}2\sqrt{4R^{-2}-1}]$, while
$I_{s,R}=(s^2+\frac12-2R^{-2},s^2)$ for 
$s\in[\frac{\sqrt3}2\sqrt{4R^{-2}-1},\frac12)$.
Hence by differentiation we obtain
\begin{align}
p_2(R)=-\frac d{dR}P_2(R)
=\frac{2\sqrt3}\pi\int_{\frac{\sqrt3}2\sqrt{4R^{-2}-1}}^{\frac12}
R^4\cdot4R^{-3}\,ds
=\frac{12}\pi\Bigl(\frac R{\sqrt3}-\sqrt{4-R^2}\Bigr).
\end{align}
Finally if $R>2$ then $I_{s,R}=\emptyset$ for all $s$,
and $J_{s,R}=\emptyset$ for $s\in(0,\frac12\sqrt{1-4R^{-2}}]$,
and $J_{s,R}=(s^2+\frac12-2R^{-2},3s^2)$ for 
$s\in[\frac12\sqrt{1-4R^{-2}},\frac12)$.
Hence by differentiation,
\begin{align}
p_2(R)
=\frac{\sqrt3}{\pi^2}\int_{\frac12\sqrt{1-4R^{-2}}}^{\frac12}
R^4\cdot 4R^{-3}\cdot
\Bigl(2\pi-6\arctan\Bigl(\sfrac{\sqrt3}2\sqrt{1-4R^{-2}}\cdot s^{-1}\Bigr)\Bigr)
\,ds,
\end{align}
and this is easily evaluated to yield the expression given in
\eqref{Ustinovs}.
Hence \eqref{Ustinovs} holds for all $R\geq0$.
\hfill$\square$

\section{Further results\label{secEx}}

We conclude by discussing a number of natural extensions and variations of Theorems \ref{Thm1} and \ref{Thm2}. They require only minor modifications in the proofs.

\subsection{Non-constant lengths}

We now admit lengths $\vecell=(\ell_1,\ldots,\ell_k)$ that depend on $n$ and $\veca$. Such a requirement may arise for instance when $C_n(\veca)$ or $C_n^+(\veca)$ is embedded in a metric space ($\RR^2$, say), and the lengths $\vecell$ are induced by the actual distance in that metric space. To make a precise statement: {\em Let $\vecell:[0,1]^k\to\RR_{\geq0}^k$ be continuous, 
and assume $\vecell(\vecx)>0$ for (Lebesgue-)almost all $\vecx\in[0,1]^k$.
Then,  for any bounded set $\scrD\subset \fF^+$ with nonempty interior and boundary of Lebesgue measure zero, we have convergence in distribution
\begin{equation}\label{Thm1eq001}
\frac{ \diam C_n^+\bigl(\vecell(n^{-1}\veca),\veca\bigr)}{\bigl(n \ell_1(n^{-1}\veca)\cdots \ell_k(n^{-1}\veca)\bigr)^{1/k} }
\xrightarrow[]{\textup{ d }}\rho(\Delta,L)
\qquad\text{as }\: T\to\infty,
\end{equation}
where the %
random variable in the left-hand side is
defined by taking $(\veca,n)$ uniformly at random in 
$\widehat\NN^{k+1}\cap T\scrD$, and the random variable in the right-hand side
is defined by taking $L$ at random in $X_k$ according to $\mu_0$.}
The analogous statement holds in the undirected case.

The limit distribution of Frobenius numbers proved in
\cite{Marklof10} can be viewed as a special case of the above result,
obtained by taking $\vecell(\vecx)\equiv\vecx$.
Indeed, for this choice of $\vecell$ we have
\begin{align}
\diam C_n^+\bigl(\vecell(n^{-1}\veca),\veca\bigr)
=n^{-1}\diam C_n^+\bigl(\veca,\veca\bigr)
=1+n^{-1}F(a_1,\ldots,a_k,n),
\end{align}
where $F(a_1,\ldots,a_k,n)$ denotes the Frobenius number of the
$k+1$ numbers $a_1,\ldots,a_k,n$; cf.\ 
\cite[Lem.\ 3]{aBjS62} or \cite[Sec.\ 2]{dBjHaNsW2005}.
Because of this relation, and since the Frobenius number is invariant
under permutation of the arguments,
\cite[Thm.\ 1]{Marklof10} follows from \eqref{Thm1eq001}.

\subsection{The distribution of distances}

Besides the diameter it is natural to consider the distribution of the distance between two randomly chosen vertices $i$ and $j$. The $\alpha$th moment (for $\alpha\in\Z_{\geq 1}$) of this distribution is
\begin{equation}
	\MM_\alpha =\frac{1}{n^2} \sum_{i,j} d(i,j)^\alpha,
\end{equation}
where $n$ is the number of vertices.
If $\Lambda$ is a sublattice of $\Z^k$ of finite index then
in view of the definition of the distance on the directed quotient
lattice graph $LG_k^+/\Lambda$, cf.\ \eqref{LGPQUOTDISTDEF}, 
we get
\begin{align}
\MM_\alpha[LG_k^+/\Lambda]=\frac1{\#(\Z^k/\Lambda)}\sum_{\vecm\in\Z^k/\Lambda}
\bigl(\min\bigl((\vecm+\Lambda)\cap\Z_{\geq0}^k\bigr)\cdot\vecell\bigr)^\alpha.
\end{align}
Similarly for the undirected quotient graph $LG_k/\Lambda$,
we get via \eqref{LGQUOTDISTDEF},
\begin{align}
\MM_\alpha[LG_k/\Lambda]=\frac1{\#(\Z^k/\Lambda)}\sum_{\vecm\in\Z^k/\Lambda}
\bigl(\min\bigl\{\vecz_+\cdot\vecell\col\vecz\in\vecm+\Lambda\bigr\}\bigr)^\alpha.
\end{align}
Following the same strategy as for the diameter one can show that under the same assumptions as in Theorem \ref{Thm1},
\begin{align}\label{DISTANCEDIRTHMRES2}
\frac{\MM_\alpha(C_n^+(\vecell,\veca))}
{(n\ell_1\cdots\ell_k)^{\alpha/k}}
\xrightarrow[]{\textup{ d }}
\int_{\R^k/L}\Psi_L(\vecy)^\alpha\,d\vecy
\qquad\text{as }\: T\to\infty,
\end{align}
where 
\begin{align}
\Psi_L(\vecy):=
\min\bigl((\vecy+L)\cap\R_{\geq0}^k\bigr)\cdot\vece.
\end{align}
Note that the scaling factor is the same as for the diameter, the maximum value of the distribution of distances; this is a non-trivial fact.
In fact joint convergence holds in \eqref{DISTANCEDIRTHMRES2} for all
$\alpha\geq1$, and from this it is possible to conclude that 
the distribution of normalized distances
$\frac{d(i,j)}{(n\ell_1\cdots\ell_k)^{1/k}}$
for vertices $i,j$ picked uniformly at random in $C_n^+(\vecell,\veca)$,
converges in distribution, as $T\to\infty$,
to the distribution of
$\Psi_L(\vecy)$ for $\vecy$ picked at random in $\R^k/L$ according
to the standard volume measure $d\vecy$.
The convergence here is in the space of probability measures on $\R_{\geq0}$,
cf., e.g.,\ \cite[Ch.\ 10]{oK97},
and the setting of the limit relation is the same as in Theorem~\ref{Thm1}.
The limiting random probability measure on $\R_{\geq0}$ obtained in this
result satisfies many interesting and beautiful properties;
we postpone a detailed discussion of these matters to a future paper.

The analogue of \eqref{DISTANCEDIRTHMRES2} in the undirected case is
\begin{align}\label{DISTANCEUTHMRES2}
\frac{\MM_\alpha(C_n(\vecell,\veca))}
{(n\ell_1\cdots\ell_k)^{\alpha/k}}
\xrightarrow[]{\textup{ d }}
\int_{\R^k/L}\tilde\Psi_L(\vecy)^\alpha\,d\vecy
\qquad\text{as }\: T\to\infty,
\end{align}
where 
\begin{align}
\tilde\Psi_L(\vecy):=
\min\bigl\{\vecz_+\cdot\vece\col\vecz\in \vecy+L\bigr\}.
\end{align}

\subsection{Shortest cycles}

The shortest cycle length (scl) of a circulant graph and its connection to the geometry of lattices is discussed in \cite{Cai99}. 
The length of the shortest cycle in a directed quotient lattice graph
$LG_k^+/\Lambda$ is
\begin{align}
\scl[LG_k^+/\Lambda]=\min\bigl(\Lambda\cap\Z_{\geq0}^k\setminus\{\bn\}\bigr)
\cdot\vecell
\end{align}
In the undirected case, there are trivial cycles 
which correspond to cycles in the covering lattice
graph $LG_k$. The shortest of these have 4 edges, and thus the girth of any quotient graph $LG_k/\Lambda$ is at most 4. We will ignore such cycles and only consider those which do \textit{not} lift to a cycle in $LG_k$, or in other words cycles which have non-zero homology when viewed as closed curves on the real torus $\R^k/\Lambda$. With this convention, the shortest length of all non-trivial cycles in a quotient lattice graph $LG_k/\Lambda$ is given by
\begin{align}
	\scl[LG_k/\Lambda]=\min\big\{ \vecm_+\cdot\vecell \col \vecm\in\Lambda\setminus\{\bn\}\bigr\}.
\end{align}

Using the same method as for the diameter one can show that,
under the same assumptions as in Theorem \ref{Thm1},
\begin{align}\label{Thm1eq003}
\frac{\scl[C_n^+(\vecell,\veca)]}{(n \ell_1\cdots \ell_k)^{1/k} }
\xrightarrow[]{\textup{ d }}
\min(L\cap\R_{\geq0}^k\setminus\{\bn\})\cdot\vece
\qquad\text{as }\: T\to\infty.
\end{align}
The complementary distribution function of the limit distribution in
this relation %
is
\begin{equation}\label{PKSCLDEF}
P_{k,\scl}(R)=\mu_0\big(\big\{L\in X_k\col R\Delta\cap L\setminus\{\vecnull\} =\emptyset \big\}\big),
\end{equation}
since for any $L\in X_k$ we have 
$\min(L\cap\R_{\geq0}^k\setminus\{\bn\})\cdot\vece>R$ if and only if
$R\Delta\cap L\setminus\{\vecnull\} =\emptyset$.
The analogue of \eqref{Thm1eq003} in the undirected case is
\begin{align}\label{SCLLIMITUNDIR}
\frac{\scl[C_n(\vecell,\veca)]}{(n \ell_1\cdots \ell_k)^{1/k} }
\xrightarrow[]{\textup{ d }}
\min\bigl\{\|\vecm\|_1\col\vecm\in L\setminus\{\bn\}\bigr\}
\qquad\text{as }\: T\to\infty,
\end{align}
and here the complementary distribution function of the limit distribution is
\begin{equation}\label{SCLLIMITUNDIR2}
\tilde P_{k,\scl}(R)=\mu_0\big(\big\{L\in X_k\col R\fP\cap L\setminus\{\vecnull\} =\emptyset \big\}\big).
\end{equation}

Comparison with \cite[Thm.\ 2.1]{partI} shows that for $k=2$ the limit distribution in the directed case,
\eqref{Thm1eq003}, \eqref{PKSCLDEF}, is related to the distribution of angles of two-dimensional lattice points (including multiplicities) via the formula
\begin{equation}
P_{2,\scl}(R)=E_{0,\vecnull}(0,\sigma)
\end{equation}
with $\sigma=R^2/2$. 
Formula (2.16) in \cite{partI} shows therefore that the density of $P_{2,\scl}(R)$ is related to the gap distribution function $P_\vecnull(s)$ for angles of lattice points,
\begin{equation}
p_{2,\scl}(R):=-\frac{d}{dR} P_{2,\scl}(R) = R\, P_\vecnull(R^2/2).
\end{equation}
An explicit formula for $P_\vecnull(s)$ can be derived from \cite{Boca00} (use Eq.\ (2.31) in \cite{partI} to relate $P_\vecnull(s)$ to $\widehat P_\vecnull(s)$; the latter is denoted $\widetilde G_{\bf D}(s)$ in \cite{Boca00}); we find
\begin{equation}
P_\vecnull(s)=\tfrac{6}{\pi^2}
\begin{cases}
1 & (0\leq s\leq\frac12) \\
s^{-1} (1+\log 2s)-1 & (\frac12\leq s\leq2) \\
s^{-1} -1 +\sqrt{1-2s^{-1}} - 2s^{-1}\log\big(\frac12\big(1+\sqrt{1-2s^{-1}}\big)\big) & (s\geq2).
\end{cases}
\end{equation}
Thus
\begin{equation}
p_{2,\scl}(R)=\tfrac{6}{\pi^2}
\begin{cases}
R & (0\leq R\leq1) \\
2R^{-1} (1+2\log R)-R & (1\leq R\leq2) \\
2R^{-1} -R +\sqrt{R^2-4} - 4R^{-1}\log\big(\frac12\big(1+\sqrt{1-4R^{-2}}\big)\big) & (R\geq2).
\end{cases}
\end{equation}

Similarly, \cite[Thm.\ 3.1]{partI} shows that for $k=2$ the limit distribution 
in the undirected case, \eqref{SCLLIMITUNDIR}, \eqref{SCLLIMITUNDIR2},
is related to the distribution of disks in random directions via the formula
\begin{equation}
\tilde P_{2,\scl}(R)=F_{0,\vecnull}(0,\sigma)
\end{equation}
with $2\sigma=R^2$. To see this, note that (in view of the $\SO(2)$ invariance of $\mu_0$) the square $\fP$ can be replaced by the square $[-R/\sqrt 2,R/\sqrt 2]^2$ which in turn (due to the invariance under the symmetry $\vecx\mapsto-\vecx$) can be replaced by the rectangle $[0,R/\sqrt 2]\times[-R/\sqrt 2,R/\sqrt 2]$. The function $F_{0,\vecnull}(0,\sigma)$ is in turn related to the free path length $\Phi_\vecnull(\xi)$ of the two-dimensional periodic Lorentz gas via formula (4.3) in \cite{partI}. This implies for the density of $\tilde P_{2,\scl}(R)$:
\begin{equation}
\tilde p_{2,\scl}(R):=-\frac{d}{dR} \tilde P_{2,\scl}(R) = R\, \Phi_\vecnull(R^2/2).
\end{equation}
The explicit formula for $\Phi_\vecnull$ in \cite{Boca03} (denoted there by $h$; the formula can also be obtained from \cite[Eqs.\ (15) and (34)]{partIII} or from \cite[Prop.\ 3 (``$r=0$'')]{aSaV2005}) yields
\begin{equation}
\tilde p_{2,\scl}(R)=\tfrac{12}{\pi^2}
\begin{cases}
R &  (0\leq R\leq1)\\
\frac{2 - R^2}{R}\big(1 + \log\big(\frac{R^2}{2 - R^2}\big)\big) & (1\leq R<\sqrt 2)\\
0 & (R\geq\sqrt 2).
\end{cases}
\end{equation}

\end{document}